\documentclass{article}[12pt]
\usepackage{amsmath, amsfonts, amscd, amssymb, amsthm, enumerate, xypic, psfrag}

\DeclareMathOperator{\WS}{\operatorname{WS}}
\DeclareMathOperator{\WA}{\operatorname{WA}}
\DeclareMathOperator{\Mat}{\operatorname{M}}
\DeclareMathOperator{\Mats}{\operatorname{S}}
\DeclareMathOperator{\Mata}{\operatorname{A}}
\DeclareMathOperator{\Diag}{\operatorname{Diag}}
\DeclareMathOperator{\GL}{\operatorname{GL}}

\DeclareMathOperator{\Vect}{\operatorname{span}}
\DeclareMathOperator{\im}{\operatorname{Im}}
\DeclareMathOperator{\tr}{\operatorname{tr}}
\DeclareMathOperator{\rk}{\operatorname{rk}}
\DeclareMathOperator{\urk}{\operatorname{urk}}

\renewcommand{\setminus}{\smallsetminus}

\def\ad{\text{ad}}
\def\F{\mathbb{F}}
\def\K{\mathbb{K}}

\def\calL{\mathcal{L}}

\def\calS{\mathcal{S}}
\def\calT{\mathcal{T}}
\def\calU{\mathcal{U}}
\def\calV{\mathcal{V}}
\def\calW{\mathcal{W}}

\def\calY{\mathcal{Y}}

\def\lcro{\mathopen{[\![}}
\def\rcro{\mathclose{]\!]}}

\theoremstyle{definition}
\newtheorem{Def}{Definition}[section]
\newtheorem{Not}[Def]{Notation}

\theoremstyle{plain}
\newtheorem{theo}{Theorem}[section]
\newtheorem{prop}[theo]{Proposition}
\newtheorem{cor}[theo]{Corollary}
\newtheorem{lemme}[theo]{Lemma}
\newtheorem{claim}{Claim}

\theoremstyle{plain}

\theoremstyle{remark}
\newtheorem{Rems}{Remarks}[section]
\newtheorem{Rem}[Rems]{Remark}

\title{Affine spaces of symmetric or alternating matrices with bounded rank}
\author{Cl\'ement de Seguins Pazzis\footnote{Universit\'e de Versailles Saint-Quentin-en-Yvelines, Laboratoire de Math\'ematiques
de Versailles, 45 avenue des Etats-Unis, 78035 Versailles cedex, France}
\footnote{e-mail address: dsp.prof@gmail.com}}

\begin{document}

\thispagestyle{plain}

\maketitle

\begin{abstract}
Let $r$ and $n$ be positive integers such that $r<n$, and $\K$ be an arbitrary field.
We determine the maximal dimension for an affine subspace of $n$ by $n$ symmetric (or alternating) matrices with entries in $\K$
and with rank less than or equal to $r$. We also classify, up to congruence, the subspaces of maximal dimension among them.
This generalizes earlier results of Meshulam, Loewy and Radwan that were previously known only for linear subspaces over fields with large cardinality
and characteristic different from $2$.
\end{abstract}



\section{Introduction}

\subsection{The problem}

Let $\K$ be a (commutative) field. We denote:
\begin{itemize}
\item By $\Mat_{n,p}(\K)$ the space of all $n$ by $p$ matrices with entries in $\K$; we also set $\Mat_n(\K):=\Mat_{n,n}(\K)$;
\item By $\Mats_n(\K)$ the space of all $n$ by $n$ symmetric matrices with entries in $\K$;
\item By $\Mata_n(\K)$ the space of all $n$ by $n$ alternating matrices with entries in $\K$
(that is, the skew-symmetric matrices with diagonal zero or, equivalently, the matrices $A \in \Mat_n(\K)$
such that $X^T AX=0$ for all $X \in \K^n$);
\item By $\GL_n(\K)$ the group of all invertible matrices of $\Mat_n(\K)$.
\end{itemize}
Given a square matrix $M=(m_{i,j}) \in \Mat_n(\K)$, we denote by
$$\Delta(M):=\begin{bmatrix}
m_{1,1} & m_{2,2} & \cdots & m_{n,n}
\end{bmatrix}^T \in \K^n$$
its diagonal vector, and by $M^\ad$ the transpose of the comatrix of $M$, also known as the classical adjoint of $M$.
Given integers $i$ and $j$ in $\lcro 1,n\rcro$, we denote by $E_{i,j}$ the elementary matrix of $\Mat_n(\K)$
with exactly one non-zero entry, located at the $(i,j)$-spot and which equals $1$.

Two subsets $\calV$ and $\calW$ of $\Mat_n(\K)$ are called \textbf{congruent} whenever
there exists a matrix $P \in \GL_n(\K)$ such that
$$\calV=P\,\calW\, P^T,$$
i.e.\ $\calV$ and $\calW$ represent the same set of bilinear forms in a different choice of basis of $\K^n$.

Given a non-empty subset $\calV$ of $\Mat_n(\K)$, we define its \textbf{upper-rank} as
$$\urk \calV:=\max \bigl\{\rk M \mid M \in \calV\bigr\}.$$

Spaces of matrices with bounded rank have attracted much scrutiny in the last decades.
In this work, we shall consider subspaces of symmetric matrices with rank less than or equal to a given integer $r$.
The corresponding problem for rectangular or square matrices has a long history dating back to Flanders
\cite{AtkLloyd,Beasley,Flanders,Meshulam,dSPboundedrank}, and the most famous result is the following one:

\begin{theo}[Flanders's theorem]\label{Flandersaffine}
Let $n \geq p$ be positive integers, and $r$ be a non-negative integer such that $r \leq \min(n,p)$.
Let $\calS$ be an affine subspace of $\Mat_{n,p}(\K)$ such that $\urk \calS \leq r$.
Then,
$$\dim \calS \leq nr.$$
Moreover, if $\dim \calS=nr$, then:
\begin{itemize}
\item Either there exists a $(p-r)$-dimensional linear subspace of $\K^p$ on which all the elements of $\calS$ vanish;
\item Or $n=p$ and there exists an $r$-dimensional linear subspace of $\K^n$ that includes the range of every element of $\calS$;
\item Or $\# \K=2$, $n=p=2$ and $\calS$ does not contain the zero matrix.
\end{itemize}
\end{theo}

See \cite{Flanders} for the original proof in a less general setting, and \cite{dSPaffpres} for the above version of the theorem.

In the symmetric case, which is a more recent issue, significant results were found by
Meshulam, Loewy and Radwan \cite{Loewy,LoewyRadwan,Meshulamsymmetric}:
some notation is necessary before we can give a proper account of them.

Given a subset $\calV$ of $\Mat_r(\K)$, with $r \leq n$, one sets
$$\widetilde{\calV}^{(n)}:=\Biggl\{ \begin{bmatrix}
S & [0]_{r \times (n-r)} \\
[0]_{(n-r) \times r} & [0]_{(n-r) \times (n-r)}
\end{bmatrix}\mid S \in \calV\Biggr\},$$
and one notes that $\urk \widetilde{\calV}^{(n)}=\urk \calV$. In particular
$\widetilde{\Mats_r(\K)}^{(n)}$ is a linear subspace of $\Mats_n(\K)$ with upper-rank $r$.

Given an even integer $r=2s$ in $\lcro 0,n\rcro$, we set
$$\WA_{n,r}(\K):=
\Biggl\{\begin{bmatrix}
A & B \\
-B^T & [0]_{(n-s) \times (n-s)}
\end{bmatrix}
\mid A \in \Mata_s(\K), \; B \in \Mat_{s,n-s}(\K)\Biggr\},$$
which is a linear subspace of $\Mata_n(\K)$ with upper-rank $r$
and dimension $\dbinom{s}{2}+s(n-s)$, and we set
$$\WS_{n,r}(\K):=
\Biggl\{\begin{bmatrix}
A & B \\
B^T & [0]_{(n-s) \times (n-s)}
\end{bmatrix}
\mid A \in \Mats_s(\K), \; B \in \Mat_{s,n-s}(\K)\Biggr\},$$
which is a linear subspace of $\Mats_n(\K)$ with upper-rank $r$
and dimension $\dbinom{s+1}{2}+s(n-s)$.
Finally, given an odd integer $r=2s+1$ in $\lcro 1,n\rcro$, we define
$\WS_{n,r}(\K)$ as the space of all symmetric matrices of the form
$$\begin{bmatrix}
[?]_{s \times s} & [?]_{s \times 1} & [?]_{s \times (n-s-1)} \\
[?]_{1 \times s} & ? & [0]_{1 \times (n-s-1)} \\
[?]_{(n-s-1) \times s} & [0]_{(n-s-1) \times 1} & [0]_{(n-s-1)\times (n-s-1)}
\end{bmatrix}.$$
One sees that $\WS_{n,r}(\K)$ is a linear subspace of $\Mats_n(\K)$ with upper-rank $r$ and dimension $\dbinom{s+1}{2}+s(n-s)+1$.

\begin{theo}[Meshulam, Loewy, Radwan]\label{generalrtheorem}
Let $S$ be a linear subspace of $\Mats_n(\K)$, and $r$ be an integer in $\lcro 1,n-1\rcro$.
Assume that $\urk S \leq r$, that $\K$ has characteristic not $2$ and that $\# \K>n$.
\begin{enumerate}[(a)]
\item If $r=2s$ for some integer $s$, then
$$\dim S \leq \max\biggl(\dbinom{r+1}{2}, \dbinom{s+1}{2}+s(n-s)\biggr).$$
\item  If $r=2s+1$ for some integer $s$, then
$$\dim S \leq \max\biggl(\dbinom{r+1}{2}, \dbinom{s+1}{2}+s(n-s)+1\biggr).$$
\end{enumerate}
In any case, equality occurs only if $S$ is congruent to $\widetilde{\Mats_r(\K)}^{(n)}$ or to $\WS_{n,r}(\K)$.
\end{theo}

In the above theorem, statements (a) and (b) were proved in \cite{Meshulamsymmetric}, whereas
the determination of the spaces of maximal dimension was established later by Loewy and Radwan \cite{LoewyRadwan}.
In Meshulam's inequality, only the assumption that $\# \K>\min(r+2,n)$ is necessary.
Our spaces $\widetilde{\Mats_r(\K)}^{(n)}$ and $\WS_{n,r}(\K)$ are denoted, respectively, by $W_1(n,r)$ and $W_2(n,r)$ in the works of the above authors.

There are quite a few possible ways to extend the above result.
One could try to understand the structure of the spaces whose upper-rank is less than or equal to $r$
and whose dimension is close to the critical one (in the same flavor as Atkinson and Lloyd's extension of
Flanders's theorem \cite{AtkLloyd}). This has been achieved by Loewy
when the critical dimension is not $\dbinom{r+1}{2}$, which encompasses the situation when $r$ is small with respect to $n$
(see \cite{Loewy}), for even values of $r$ only. In that case,
the natural result states that, if the dimension of a linear subspace $S$ with upper-rank bounded above by $r$
is close enough to the maximal one, then there should be a linear subspace of dimension $n-\frac{r}{2}$ of $\K^n$
on which all the matrices of $S$ are totally isotropic (which is the situation for the $\WS_{n,r/2}(\K)$ space).

Another desirable improvement over the above theorem would be the removal of the cardinality assumption
and of the characteristic assumption.

Finally, one could try to extend these results to affine subspaces
as well. Besides being a natural question, such an extension is motivated by potential applications
and has already proved fruitful in the situation of spaces of rectangular matrices with bounded rank
(note how the main result of \cite{dSPaffpres} is used in a crucial way in \cite{dSPboundedrank}).
Here is one such application: in the above result of Meshulam, the dimensional inequality can be reformulated as stating that if
the dimension of a linear subspace $\calV$ of symmetric matrices is large enough, then this subspace
must contain a matrix with rank greater than $r$. Now, say that we want to know if $\calV$
is actually \emph{spanned} by its matrices with rank greater than $r$. If this is not the case
then some affine hyperplane of $\calV$ would contain no such matrix, and hence the affine equivalent of Meshulam's theorem
would lead to a contradiction should the dimension of $\calV$ be large enough.
Moreover, as we shall see, extending the framework to affine subspaces is key to the study of the
special case when $\# \K=2$.

Note however that the extension to affine subspaces is trivial when $\# \K>r$: if in that case
we consider an affine subspace $\calS$ of $\Mat_n(\K)$ with upper-rank less than $r$, then one checks that $\Vect(\calS)$
is a linear subspace of $\Mat_n(\K)$ with upper-rank less than $r$;
indeed, $\calS$ is included in $\{M \in \Mat_n(\K) : \rk M \leq r\}$, which is defined by a system of $(r+1)$-homogeneous polynomial equations on $\Mat_n(\K)$; in general if a $(r+1)$-homogeneous polynomial function $f : \Mat_n(\K) \rightarrow \K$ vanishes everywhere on $\calS$,
then it must also vanish everywhere on $S$ since $\K$ has more than $r$ elements, and hence it vanishes everywhere on $\Vect(\calS)$.

\subsection{Main results}\label{section1.2}

In this article, we shall extend Theorem \ref{generalrtheorem} to an arbitrary field and to affine subspaces,
with the characteristic $2$ case taken into account. We shall also prove a similar result for spaces of alternating matrices with
bounded rank. We will not try to classify spaces whose dimension is close to the maximal one, but we are confident
that a proper use of our new techniques will help us make substantial advances in that direction in the near future.

Before we state our results, it is necessary to make a few comments on the characteristic $2$ case:
first of all, if $\K$ has characteristic $2$ then every alternating matrix over $\K$ is also symmetric, and hence
for every even integer $r$ the set $\Mata_{r+1}(\K)$ is a linear subspace of symmetric matrices with upper rank $r$;
hence, for all $n \geq r+1$, the set $\widetilde{\Mata_{r+1}(\K)}^{(n)}$ is a linear subspace of $\Mats_n(\K)$
with upper rank $r$ and dimension $\dbinom{r+1}{2}$.

Next, in the situation where $\# \K=2$, we can give an extra general class of large affine spaces
of singular symmetric matrices:
we define
$$Z_n(\K):=\Biggl\{\begin{bmatrix}
S & \Delta(S) \\
\Delta(S)^T & (n-1).1_\K
\end{bmatrix} \mid S \in \Mats_{n-1}(\K)\Biggr\}.$$
Obviously, $Z_n(\K)$ is an affine subspace of $\Mats_n(\K)$ with dimension $\dbinom{n}{2}$
(note that it is a linear subspace if and only if $n$ is odd).
Moreover, every matrix in $Z_n(\K)$ is singular. To see this, let $S \in \Mats_{n-1}(\K)$ and set
$M:=\begin{bmatrix}
S & \Delta(S) \\
\Delta(S)^T & (n-1).1_\K
\end{bmatrix}$; then,
$$\det M=(n-1)\det S-\Delta(S)^T S^\ad \Delta(S),$$
and as $\# \K=2$ one sees that
\begin{align*}
\Delta(S)^T S^\ad \Delta(S) & =\Delta(S)^T \Delta(S^\ad) \qquad \text{(since $\K=\{0,1\}$)} \\
& = \tr(S S^\ad) \qquad \text{(since $\K$ has characteristic $2$ and $S$ and $S^\ad$ are symmetric)} \\
& = \tr(\det(S).I_{n-1})=(n-1).\det S,
\end{align*}
whence $\det M=0$.
On the other hand, it is obvious that $Z_n(\K)$ contains a matrix with rank $n-1$, and hence its upper-rank is exactly $n-1$.

It follows that for all positive integers $r<n$, the space $\widetilde{Z_{r+1}(\K)}^{(n)}$ is an affine subspace of
$\Mats_n(\K)$ with upper rank $r$ and dimension $\dbinom{r+1}{2}$.

Still assuming that $\# \K=2$, we define, for every odd integer $r=2s+1$ such that $r<n$, the set
$Z'_{n,r}(\K)$ of all symmetric matrices of the form
$$\begin{bmatrix}
[?]_{s \times s} & [?]_{s \times 2} & [?]_{s \times (n-s-2)} \\
[?]_{2 \times s} & A & [0]_{2 \times (n-s-2)} \\
[?]_{(n-s-2) \times s} & [0]_{(n-s-2) \times 2} & [0]_{(n-s-2) \times (n-s-2)}
\end{bmatrix} \quad \text{with $A \in Z_2(\K)$.}$$
As every matrix of $Z_2(\K)$ has rank $1$ one checks that the upper-rank of $Z'_{n,r}(\K)$
equals $r$. Moreover, one checks that $\dim Z'_{n,r}(\K)=\dbinom{s+1}{2}+s(n-s)+1$.

Finally, if $\# \K=2$ then we can give three additional interesting examples of $3$-dimensional affine subspaces
of $\Mats_3(\K)$ consisting of singular matrices only:
$$\calY_1(\K):=\Biggl\{\begin{bmatrix}
a & b & c+1 \\
b & c & 0 \\
c+1 & 0 & 0
\end{bmatrix} \mid (a,b,c)\in \K^3\Biggr\},$$
$$\calY_2(\K):=\Biggl\{\begin{bmatrix}
a & b & a+b+c+1 \\
b & c & 0 \\
a+b+c+1 & 0 & c
\end{bmatrix} \mid (a,b,c)\in \K^3\Biggr\},$$
and
$$\calY_3(\K):=\Biggl\{\begin{bmatrix}
a & b & c \\
b & 0 & a+1 \\
c & a+1 & 0
\end{bmatrix} \mid (a,b,c)\in \K^3\Biggr\}.$$
In each case, one uses the identity $\forall x \in \K, \; x^2=x$ to obtain that the determinant of any matrix in the given space
equals $0$. Note that $\calY_1(\K),\calY_2(\K),\calY_3(\K)$ are all non-linear
affine subspaces of $\Mats_3(\K)$.

Finally, still assuming that $\# \K=2$, we have an exceptional affine subspace of $\Mata_4(\K)$
with upper-rank $2$ and dimension $3$:
$$\calU(\K):=\left\{  \begin{bmatrix}
0 & a & b & c+1 \\
a & 0 & c & 0 \\
b & c & 0 & 0 \\
c+1 & 0 & 0 & 0
\end{bmatrix} \mid (a,b,c) \in \K^3\right\}.$$
Indeed, one computes that
$$\forall (a,b,c)\in \K^3, \quad \begin{vmatrix}
0 & a & b & c+1 \\
a & 0 & c & 0 \\
b & c & 0 & 0 \\
c+1 & 0 & 0 & 0
\end{vmatrix}=(c(c+1))^2=0.$$

\vskip 3mm
Now, we are finally ready to state our results.
We shall start with the alternating matrices, for which there are fewer special cases.

\begin{Not}
Given non-negative integers $n$ and $r$ such that $r<n$ and $r=2s$ for some integer $s$,
we set
$$a_{n,r}^{(1)}:=\dbinom{r+1}{2} \quad \text{and} \quad a_{n,r}^{(2)}:=\dbinom{s}{2}+s(n-s).$$
\end{Not}

\begin{Rem}\label{dimensionremarkalternating}
One checks that
$$\max\bigl(a_{n,2s}^{(1)},a_{n,2s}^{(2)}\bigr)=\begin{cases}
a_{n,2s}^{(1)} & \text{if and only if $5s \geq 2n-3$ or $s=0$} \\
a_{n,2s}^{(2)} & \text{if and only if $5s \leq 2n-3$ or $s=0$.}
\end{cases}$$
\end{Rem}

\begin{theo}[Classification theorem for spaces of alternating matrices]\label{maintheoalternate}
Let $\K$ be an arbitrary field, and let $n$ and $s$ be non-negative integers with $2s<n$.
Let $\calS$ be an affine subspace of $\Mata_n(\K)$ such that $\urk \calS \leq 2s$.
Then,
$$\dim \calS \leq \max\bigl(a_{n,2s}^{(1)},a_{n,2s}^{(2)}\bigr).$$
Moreover, if equality holds then:
\begin{enumerate}[(a)]
\item Either $\calS$ is congruent to $\WA_{n,2s}(\K)$;
\item Or $\calS$ is congruent to $\widetilde{\Mata_{2s+1}(\K)}^{(n)}$;
\item Or $s=1$, $n=4$, $\# \K=2$ and $\calS$ is congruent to $\calU(\K)$.
\end{enumerate}
\end{theo}

For sufficiently large fields, the inequality statement from this theorem
was already known (see Remark 1 in \cite{Meshulamsymmetric} -- with a misprint -- or Theorem 1.2 from \cite{GelbordMeshulam} for $p=2$).

It is obvious that the three given cases are pairwise incompatible provided that $s>0$.
Indeed, on the one hand $\calU(\K)$ does not contain the zero matrix, whereas both spaces
$\WA_{n,2s}(\K)$ and $\widetilde{\Mata_{2s+1}}^{(n)}$ do. On the other hand,
there is no non-zero vector of $\K^n$ on which all the matrices of $\WA_{n,2s}(\K)$ vanish,
which yields that $\WA_{n,2s}(\K)$ is not congruent to $\widetilde{\Mata_{2s+1}(\K)}^{(n)}$ if $2s+1<n$;
if $n=2s+1$ and $s>0$ then $\WA_{n,2s}(\K) \subsetneq \Mata_n(\K)=\widetilde{\Mata_{2s+1}(\K)}^{(n)}$ whence
$\dim \WA_{n,2s}(\K)<\dim \widetilde{\Mata_{2s+1}(\K)}^{(n)}$.

Next, we state the corresponding result for spaces of symmetric matrices.
It is somewhat more complicated, due to the characteristic $2$ case.

\begin{Not}
Given non-negative integers $n$ and $r$ such that $r<n$,
we set
$$s^{(1)}_{n,r}:=\dbinom{r+1}{2}$$
and
$$s^{(2)}_{n,r}:=\begin{cases}
\dbinom{s+1}{2}+s(n-s) & \text{if $r=2s$ is even} \\
\dbinom{s+1}{2}+s(n-s)+1 & \text{if $r=2s+1$ is odd.}
\end{cases}
$$
\end{Not}

\begin{Rem}\label{dimensionremarksymmetric}
One checks that
$$\max\biggl(\dbinom{2s+1}{2}, \dbinom{s+1}{2}+s(n-s)\biggr)=\begin{cases}
\dbinom{2s+1}{2} & \text{if and only if $5s \geq 2n-1$ or $s=0$} \\
\dbinom{s+1}{2}+s(n-s) & \text{if and only if $5s \leq 2n-1$ or $s=0$}
\end{cases}$$
and
$$\max\biggl(\dbinom{2s+2}{2}, \dbinom{s+1}{2}+s(n-s)+1\biggr)=\begin{cases}
\dbinom{2s+2}{2} & \text{if and only if $5s \geq 2n-5$ or $s=0$} \\
1+\dbinom{s+1}{2}+s(n-s) & \text{if and only if $5s \leq 2n-5$ or $s=0$.}
\end{cases}$$

In particular, if $n>3$ then $s_{n,n-1}^{(1)}>s_{n,n-1}^{(2)}$.
\end{Rem}

\begin{theo}[Classification theorem for spaces of symmetric matrices]\label{maintheosymmetric}
Let $\K$ be an arbitrary field, and let $n$ and $r$ be non-negative integers with $r< n$.
Let $\calS$ be an affine subspace of $\Mats_n(\K)$ such that $\urk \calS \leq r$.
Then,
$$\dim \calS \leq \max\bigl(s^{(1)}_{n,r},s^{(2)}_{n,r}\bigr).$$
If equality holds, then one of the following situations holds:
\begin{enumerate}[(i)]
\item $\calS$ is congruent to $\widetilde{\calS_r(\K)}^{(n)}$;
\item $\calS$ is congruent to $\WS_{n,r}(\K)$;
\item $\K$ has characteristic $2$, $r$ is even and
$\calS$ is congruent to $\widetilde{\Mata_{r+1}(\K)}^{(n)}$;
\item $\K$ has cardinality $2$ and $\calS$ is congruent to $\widetilde{Z_{r+1}(\K)}^{(n)}$;
\item $\K$ has cardinality $2$, $r$ is odd and $\calS$ is congruent to $Z'_{n,r}(\K)$;
\item $\K$ has cardinality $2$, $r=2$, $n=3$ and $\calS$ is congruent to one of the affine spaces
$\calY_1(\K)$, $\calY_2(\K)$ and $\calY_3(\K)$.
\end{enumerate}
\end{theo}

Again, in that theorem all the given cases are pairwise incompatible provided that $r>0$, unless $r=1$, in which case $Z'_{n,r}(\K)=\widetilde{Z_{r+1}(\K)}^{(n)}$ and
$\WS_{n,r}(\K)=\widetilde{\Mats_1(\K)}^{(n)}$: we shall now demonstrate this.
\begin{itemize}
\item In any case but case (iii), the space $\calS$ must contain a non-alternating matrix, and hence case (iii)
is incompatible with all the other ones.
\item In case (v), $\calS$ does not contain the zero matrix, in contrast with case (ii). Hence, cases (ii) and (v) are incompatible.
\item The vector space $\{X \in \K^n : \; \forall M \in \calS, \; MX=0\}$ has dimension $n-r$ in case (i),
dimension $0$ in cases (ii), dimension $0$ in case (v) unless $r=1$ (in which case it has dimension $n-2$), and dimension
$n-r-1$ in case (iv). Hence, cases (i), (ii) and (v) are pairwise incompatible, and case (iv) is incompatible with case (i).
Moreover, if $n>r+1$ and $r>1$ then case (iv) is incompatible with cases (ii) and (v).

\item Assume that $n=r+1$, $r>1$ and $\# \K=2$. If $\widetilde{Z_{r+1}(\K)}^{(n)}$ were congruent to $Z'_{n,r}(\K)$ with $r$ odd, we would obtain $n \leq 3$ by Remark \ref{dimensionremarksymmetric}, which would contradict $1<r<n$.
Assume now that $\widetilde{Z_{r+1}(\K)}^{(n)}$ is congruent to $\WS_{n,r}(\K)$. Then, again $n \leq 3$, and $r$
must be even because $\widetilde{Z_{r+1}(\K)}^{(n)}$ must contain the zero matrix. Hence,
$n=3$ and $r=2$, and $Z_3(\K)$ would be congruent to $\WS_{3,2}(\K)$.
Yet, $Z_3(\K)$ contains only two alternating matrices, whereas $\WS_{3,2}(\K)$ contains four of them.
Hence, $Z_3(\K)$ is not congruent to $\WS_{3,2}(\K)$.

We conclude that cases (i) to (v) are pairwise incompatible, unless $r=1$ in which situation cases (iv) and (v) are equivalent and
cases (i) and (ii) are equivalent.
\item Assume now that $\# \K=2$ and $(n,r)=(3,2)$. Then, case (v) cannot occur, and, in cases (i) to (iv),
$\calS$ contains the zero matrix. Thus, cases (i) to (v) are incompatible with case (vi).
\end{itemize}
Finally, assuming that $\# \K=2$, let us prove that $\calY_1(\K)$, $\calY_2(\K)$ and $\calY_3(\K)$ are pairwise non-congruent.
One checks that $\calY_1(\K)$ contains only two rank $1$ matrices (namely, $E_{2,2}$ and $E_{1,1}+E_{1,2}+E_{2,1}+E_{2,2}$),
whereas $\calY_2(\K)$ and $\calY_3(\K)$ only contain one rank $1$ matrix (namely, $E_{1,1}$).
Finally, $\calY_2(\K)$ contains two alternating matrices, whereas $\calY_3(\K)$ contains four,
which proves that they are non-congruent.

\subsection{Main strategy}

As was the case in a lot of recent research on similar topics, our results will be obtained
by induction on both $n$ and $r$. To make the induction process work,
the major key consists in the study of the structure of the subset of matrices with rank $1$ or $2$ in the translation vector space $S$ of
the given affine space $\calS$ of bounded rank symmetric or alternating matrices.
This motivates the following notation:

\begin{Not}
Given a linear hyperplane $H$ of $\K^n$ and a subset
$V$ of $\Mat_n(\K)$, we denote by $V_H$ the set of all matrices $M \in V$
such that
$$\forall (X,Y) \in H^2, \; X^T M Y=0,$$
that is the set of all matrices of $V$ for which $H$ is totally singular.
\end{Not}

To get a clear picture, if $H=\K^{n-1} \times \{0\}$ then $V_H$ consists of all the matrices of $V$ of the following form:
$$\begin{bmatrix}
[0]_{(n-1) \times (n-1)} & [?]_{(n-1) \times 1} \\
[?]_{1 \times (n-1)} & ?
\end{bmatrix}.$$

With those sets, we have a way of differentiating between $\WS_{n,2s}(\K)$ and $\widetilde{\Mats_{2s}(\K)}^{(n)}$ when $2s<n$:
for the first one, we have $\dim V_H \geq s$ for every linear hyperplane $H$ of $\K^n$, whereas for the second one the linear hyperplane
$H:=\K^{n-1} \times \{0\}$ satisfies $\dim V_H=0$.

The first -- crucial -- step, both for the proof of the inequality and for the study of the case of equality in Theorems \ref{maintheoalternate} and \ref{maintheosymmetric}, consists in finding a linear hyperplane $H$ of $\K^n$ for which the dimension of $S_H$ is small.

\vskip 3mm
Say that $H=\K^{n-1} \times \{0\}$, and let $\calS$ be an affine subspace of $\Mats_n(\K)$ with upper-rank at most $r$,
where $\K$ is a field of characteristic not $2$.

Assume first that $S_H=\{0\}$. Then, we can split every matrix $M$ of $\calS$ up as
$$M=\begin{bmatrix}
P(M) & [?]_{(n-1) \times 1} \\
[?]_{1 \times (n-1)} & ?
\end{bmatrix} \quad \text{with $P(M) \in \Mats_{n-1}(\K)$.}$$
Then, $P(\calS)$ is an affine subspace of $\Mats_{n-1}(\K)$ and $\dim \calS=\dim P(\calS)$.
Moreover $\urk P(\calS) \leq \urk \calS$.
By induction on $n$ we can hope to obtain an upper-bound for the dimension of $\calS$.

Next, assume that $S_H$ contains a rank $2$ matrix, say $E_{1,n}+E_{n,1}$ (in the characteristic $2$ case, not all rank $2$ matrices in $S_H$
are congruent to such a matrix, but let us not get distracted by this side issue).
Then, splitting any matrix $M$ of $\calS$ up as
$$M=\begin{bmatrix}
? & [?]_{1 \times (n-2)} & ? \\
[?]_{(n-2) \times 1} & K(M) & [?]_{(n-2) \times 1} \\
? & [?]_{1 \times (n-2)} & ?
\end{bmatrix} \quad \text{with $K(M) \in \Mats_{n-2}(\K)$,}$$
one can prove that $\urk K(\calS) \leq \urk \calS -2$ (see Lemma \ref{symmetriccornerlemma1}), whereas by the rank theorem
$$\dim \calS \leq \dim K(\calS)+(n-1)+\dim S_H.$$
Thus, if we have a good enough upper-bound on the dimension of $S_H$, we can hope to
get the desired outcome by induction on the size of the matrices.
In this prospect, it is of much interest to note that $s_{n,2s}^{(2)}=s_{n,2(s-1)}^{(2)}+(n-1)+s$ and
$s_{n,2s+1}^{(2)}=s_{n,2(s-1)+1}^{(2)}+(n-1)+s$.
In general, we shall try to find $H$ such that $\dim S_H \leq s$ where $s:=\left \lfloor \frac{r}{2}\right\rfloor$.

Finally, assume that $S_H$ contains a rank $1$ matrix, that is $S_H$ contains $E_{n,n}$.
Then, one proves that $\urk P(\calS) \leq \urk \calS-1$ (see Lemma \ref{symmetriccornerlemma2}),
and by the rank theorem $\dim \calS=\dim P(\calS)+\dim S_H$. Then, provided that the dimension of $S_H$ is very small
we can, once more, hope to prove the inequality statement by induction on $n$ and $r$.

\vskip 3mm
The study of spaces with maximal dimension is performed in essentially the same way.
Additional techniques are required there to ``lift" the structure of the extracted block space
(either $P(\calS)$ or $K(\calS)$, depending on the structure of the $S_H$ space)
in order to understand the structure of the whole space $\calS$. We shall be confronted with two main situations.

Suppose first that $S_H=\{0\}$. Then, we will use the induction hypothesis to demonstrate that $P(\calS)$ is congruent to $\widetilde{\Mats_r(\K)}^{(n-1)}$.
We can actually assume that $P(\calS)=\widetilde{\Mats_r(\K)}^{(n-1)}$ and
we want to show that $\calS$ is congruent to $\widetilde{\Mats_r(\K)}^{(n)}$.
We find affine maps $C_1 : \Mats_r(\K) \rightarrow \K^r$, $C_2 : \Mats_r(\K) \rightarrow \K^{n-r-1}$ and $b : \Mats_r(\K) \rightarrow \K$
such that $\calS$ is the space of all matrices of the form
$$\begin{bmatrix}
N & [0]_{r \times (n-r-1)} & C_1(N) \\
[0]_{(n-r-1) \times r} & [0]_{(n-r-1) \times (n-r-1)} & C_2(N) \\
C_1(N)^T & C_2(N)^T & b(N)
\end{bmatrix} \quad \text{with $N \in \Mats_r(\K)$.}$$
We will easily obtain that $C_2=0$. Then, we shall demonstrate that $C_1$ maps every matrix of $\Mats_r(\K)$
to a vector of its range, in other words it is \emph{range-compatible} (see \cite{dSPRC1}). Using recent theorems on range-compatible maps,
we shall deduce that $C_1 : N \mapsto NY$ for some fixed vector $Y \in \K^r$, and thanks to an additional congruence transformation
we shall reduce the situation to the one where $C_1=0$. In that situation it will be easy to obtain that $b=0$
and to conclude that $\calS=\widetilde{\Mats_r(\K)}^{(n)}$.

Next, supposing that $S_H \neq \{0\}$, $\dim S_H \leq s$ and $\K$ has characteristic not $2$,
we will prove that $P(\calS)$ is congruent to $\WS_{n-1,r}(\K)$,
and we will lose no generality in assuming that $P(\calS)=\WS_{n-1,r}(\K)$.
Then, we want to prove that $\calS$ is congruent to $\WS_{n,r}(\K)$.
Here, the affine version of Flanders's theorem will play a major part!
To fix the ideas, say that $r=2s$ and that $P(\calS)=\WS_{n-1,r}(\K)$.
Then, every matrix $M$ of $\calS$ splits as
$$M=\begin{bmatrix}
[?]_{s \times s} & B(M)^T & [?]_{s \times 1} \\
B(M) & [0]_{(n-s-1) \times (n-s-1)} & C(M) \\
[?]_{1 \times s} & C(M)^T & ?
\end{bmatrix}$$
with $B(M) \in \Mat_{n-s-1,s}(\K)$ and $C(M) \in \K^s$.
Then, we aim at proving that, after performing a well-chosen congruence transformation on $\calS$, one can assume that $C=0$.
To achieve this, we will prove that there exists a vector $Y \in \K^s$ such that
$\forall M \in \calS, \; C(M)=B(M)Y$. It was assumed that
$P(\calS)=\WS_{n-1,2s}(\K)$, whence $B(\calS)=\Mat_{n-s-1,s}(\K)$. Then,
we shall remark that every matrix in the affine space
$$\calT:=\Bigl\{\begin{bmatrix}
B(M) & C(M)
\end{bmatrix} \mid M \in \calS\Bigr\}$$
has rank less than or equal to $s$ (this is easily seen from the above form of the matrices in $\calS$
and from the fact that $\urk \calS \leq 2s$).
The vector $Y$ will be obtained by applying Flanders's theorem to the affine space $\calT$.
Once we have reduced the situation to the case when $C=0$, it will be an easy task to prove that the entry in the lower-right corner
is systematically zero for the matrices in $\calS$, yielding that $\calS \subset \WS_{n,2s}(\K)$.
Then, the conclusion will follow by remarking that both spaces must have the same dimension.

\subsection{Structure of the article}

The article is laid out as follows.

In Section \ref{technicalsection}, we set all the basic tools that are needed to solve our problem.
The first one consists in the extraction lemmas which help majorize the rank of specific submatrices of $\calS$ when we have
a matrix with rank $1$ or $2$ in the translation vector space of $\calS$ (Section \ref{extractionsection}).
The next set of results deals with the $S_H$ spaces (Section \ref{rank2section}): we shall prove that in most cases there
exists an $S_H$ space with dimension less than or equal to $s$, where $s:=\left \lfloor \frac{\urk \calS}{2}\right\rfloor$.
Here, most of the difficulty lies in the characteristic $2$ case for symmetric matrices when $\urk \calS=n-1$.
Section \ref{RCsection} consists of a quick review of the known results on range-compatible linear maps
on the spaces $\Mats_n(\K)$ and $\Mata_n(\K)$ (these results are needed to analyze the
spaces of maximal dimension in the event when $S_H=\{0\}$ for some linear hyperplane $H$).
In Section \ref{rankatmost1section}, we shall classify all the $1$-dimensional affine subspaces of symmetric matrices
with rank at most $1$, a result which is the first step in our proof of Theorem \ref{maintheosymmetric}
and which will be used frequently. Finally, in Section \ref{Flanderscorsection}
we prove an important corollary to Flanders's theorem on affine subspaces: this result will be of great use
to analyze the spaces of maximal dimension in the event when there is a linear hyperplane $H$ such that
$0<\dim S_H \leq \bigl\lfloor \frac{\urk \calS}{2}\bigr\rfloor$.

The next six sections are devoted to the proofs of Theorems \ref{maintheoalternate} and \ref{maintheosymmetric}.
We shall always tackle the problems in the increasing order of difficulty, hence always starting with the alternating case,
which features the least amount of technical difficulties, and always ending with the symmetric case over fields with characteristic $2$,
by far the most involving. As far as the inequality statements are concerned, we will first prove
the one in Theorem \ref{maintheoalternate} (Section \ref{section3}),
then the one in Theorem \ref{maintheosymmetric} over fields with characteristic not $2$ (Section \ref{section4}),
and finally the one in Theorem \ref{maintheosymmetric} over fields with characteristic $2$ (Section \ref{section5}).
Using a similar pattern, we shall classify the spaces with maximal dimension first in the alternating case (Section \ref{section6}),
then in the symmetric case over fields with characteristic not $2$ (Section \ref{section7}) and finally over fields
with characteristic $2$ (Section \ref{section8}). In the latter case, most of the difficulty comes
from fields with two elements, and substantial shortcuts (which we will not discuss here)
could be obtained by discarding such fields.

\section{Main technical tools}\label{technicalsection}

\subsection{The canonical situation}\label{extractionsection}

We start with a well-known lemma on the Schur complement.

\begin{lemme}\label{schurcomplementlemma}
Let $r \in \lcro 1,n-1\rcro$, $A \in \GL_r(\K)$, $B \in \Mat_{n-r,r}(\K)$, $C \in \Mat_{r,n-r}(\K)$ and $D \in \Mat_{n-r}(\K)$.
Then,
$$\rk \begin{bmatrix}
A & C \\
B & D
\end{bmatrix} = r+\rk(B A^{-1} C-D).$$
\end{lemme}

\begin{proof}
Indeed, by Gaussian elimination we have
$$\rk \begin{bmatrix}
A & C \\
B & D
\end{bmatrix} =\rk \begin{bmatrix}
A & C \\
0 & D-BA^{-1}C
\end{bmatrix},$$
and the result follows from the fact that $\rk A=r$.
\end{proof}

\begin{lemme}\label{alternatingdeterminantlemma}
Let $A \in \Mata_n(\K)$ and $C \in \K^{n-1}$. Let us split up
$$A=\begin{bmatrix}
P & C_0 \\
-C_0^T & 0
\end{bmatrix} \quad \text{with $P \in \Mata_{n-1}(\K)$ and $C_0 \in \K^{n-1}$,}$$
and set
$$N:=\begin{bmatrix}
[0]_{(n-1) \times (n-1)} & C \\
-C^T & 0
\end{bmatrix}.$$
Assume that $A+tN$ is singular for all $t \in \K$. Then,
$$C^T P^\ad C=0.$$
\end{lemme}

\begin{proof}
Let $t \in \K$.
Computing the determinant of $A+tN$ yields
$$(t\, C+C_0)^T P^\ad (t\, C+C_0)^T=0.$$
Subtracting the case $t=0$ yields
$$\forall t \in \K, \; \bigl(C^T P^\ad C\bigr)\, t^2+\bigl(C_0^T P^\ad C+C^T P^\ad C_0\bigr)\, t=0$$
If $\# \K>2$, we immediately deduce that $C^T P^\ad C=0$. \\
If $\# \K=2$, then $P$ is symmetric and hence $P^\ad$ is symmetric, which leads to
$C_0^T P^\ad C+C^T P^\ad C_0=0$; the case $t=1$ in the above identity then yields $C^T P^\ad C=0$.
\end{proof}

\begin{lemme}\label{symmetricdeterminantlemma}
Let $A \in \Mats_n(\K)$, $C \in \K^{n-1}$ and $a \in \K$. Let us split up
$$A=\begin{bmatrix}
P & C_0 \\
C_0^T & a_0
\end{bmatrix} \quad \text{with $P \in \Mats_{n-1}(\K)$, $C_0 \in \K^{n-1}$ and $a_0 \in \K$,}$$
and set
$$N:=\begin{bmatrix}
[0]_{(n-1) \times (n-1)} & C \\
-C^T & a
\end{bmatrix}.$$
Assume that $A+tN$ is singular for all $t \in \K$.
\begin{enumerate}[(a)]
\item If $\# \K>2$ then
$$C^T P^\ad C=0.$$
\item If $\K$ has characteristic $2$
then
$$C^T P^\ad C=a \det P.$$
\item If $C=0$ and $a \neq 0$ then $\det P=0$.
\end{enumerate}
\end{lemme}

\begin{proof}
As in the preceding lemma, computing determinants yields
$$\forall t \in \K, \; (t\,a+a_0)\det P=(t\,C+C_0)^T P^\ad (t\,C+C_0).$$
Since $P$ is symmetric, so is $P^\ad$.
Hence, subtracting the special case $t=0$ leads to
$$\forall t \in \K, \; \bigl(C^T P^\ad C\bigr)\,t^2+\bigl(2 C_0^T P^\ad C-a \det P\bigr)\,t=0.$$
If $\# \K>2$, this yields $C^T P^\ad C=0$.
If $\K$ has characteristic $2$, the case $t=1$ yields $C^T P^\ad C=a \det P$. \\
In any case, if $C=0$ and $a \neq 0$ then the case $t=1$ yields $\det P=0$.
\end{proof}

From those two lemmas, a handful of other ones can be deduced:

\begin{lemme}\label{alternatingcornerlemma}
Let $n$ be an integer such that $n \geq 3$. Let $r$ be an even positive integer such that $r<n$.
Let $A \in \Mata_n(\K)$, which we split up as
$$A=\begin{bmatrix}
0 & L & a \\
-L^T & B & C \\
-a & -C^T & 0
\end{bmatrix}$$
with $L \in \Mat_{1,n-2}(\K)$, $C \in \K^{n-2}$, $a \in \K$ and $B \in \Mata_{n-2}(\K)$.
Set $N:=E_{1,n}-E_{n,1}$ and assume that $\rk(A+tN) \leq r$ for all $t \in \K$. \\
Then, $\rk B \leq r-2$.
\end{lemme}

\begin{proof}
Set $s:=\rk B$ and assume that $s>r-2$.
Then, there are invertible matrices $Q \in \GL_{n-2}(\K)$ and $B' \in \Mata_s(\K) \cap \GL_s(\K)$
such that $QBQ^T=B' \oplus 0_{n-2-s}$.
Setting $P:=I_1 \oplus Q \oplus I_1$, we see that $PNP^T=N$
and
$$P A P^T=\begin{bmatrix}
0 & L' & a \\
-(L')^T & QBQ^T & C' \\
-a & -(C')^T & 0
\end{bmatrix} \quad \text{where $L':=LQ^T$ and $C':=QC$.}$$
Thus, no generality is lost in assuming that $B=B' \oplus 0_{n-2-s}$.
Writing $L=\begin{bmatrix}
L_1 & L_2
\end{bmatrix}$, where $(L_1,L_2) \in \Mat_{1,s}(\K) \times \Mat_{1,n-2-s}(\K)$ and
$C=\begin{bmatrix}
C_1 \\
C_2
\end{bmatrix}$, where $C_1 \in \K^s$ and $C_2 \in \K^{n-2-s}$, we
can set
$$A':=\begin{bmatrix}
0 & L_1 & a \\
-L_1^T & B' & C_1 \\
-a & -C_1^T & 0
\end{bmatrix} \quad \text{and} \quad
N':=\begin{bmatrix}
0 & [0]_{1 \times s} & 1 \\
[0]_{s \times 1} & [0]_{s \times s} & [0]_{s \times 1} \\
-1 & [0]_{1 \times s} & 0
\end{bmatrix}$$
and we learn that $A'+tN'$ is singular for all $t \in \K$ since it is an $(s+2)$ by $(s+2)$ submatrix of $A+tN$ with $s+2>r$.

Using Lemma \ref{alternatingdeterminantlemma}, we deduce that $\det B'=0$, which contradicts the fact that $B'$ is non-singular.
Therefore, $s \leq r-2$, as claimed.
\end{proof}

With the same line of reasoning, we obtain the following results, this time by applying Lemma \ref{symmetricdeterminantlemma}:
\begin{lemme}\label{symmetriccornerlemma1}
Let $n$ be an integer such that $n \geq 3$. Let $r$ be a positive integer such that $r<n$.
Let $A \in \Mats_n(\K)$, which we split up as
$$A=\begin{bmatrix}
a & L & b \\
L^T & B & C \\
b & C^T & c
\end{bmatrix}$$
with $L \in \Mat_{1,n-2}(\K)$, $C \in \K^{n-2}$, $(a,b,c) \in \K^3$ and $B \in \Mats_{n-2}(\K)$.
Let $d \in \K$, set
$$N:=\begin{bmatrix}
0 & [0]_{1 \times (n-2)} & 1 \\
[0]_{(n-2) \times 1} & [0]_{(n-2) \times (n-2)} & [0]_{(n-2) \times 1} \\
1 & [0]_{1 \times (n-2)} & d
\end{bmatrix}$$
and assume that $\rk(A+tN) \leq r$ for all $t \in \K$. \\
If $\# \K>2$ or $d=0$, then $\rk B \leq r-2$.
\end{lemme}

\begin{lemme}\label{symmetriccornerlemma2}
Let $n$ be an integer such that $n \geq 2$. Let $r$ be a positive integer such that $r<n$.
Let $A \in \Mats_n(\K)$, which we split up as
$$A=\begin{bmatrix}
P & C \\
C^T & a
\end{bmatrix} \quad \text{with $P \in \Mats_{n-1}(\K)$, $C \in \K^{n-1}$ and $a \in \K$.}$$
If $\rk(A+tE_{n,n})\leq r$ for all $t \in \K$, then $\rk P \leq r-1$.
\end{lemme}

Applying congruence transformations yields the following corollary:

\begin{cor}\label{extractioncorollary}
Let $(i,j)\in \lcro 1,n\rcro^2$ be such that $i \neq j$, and let $M \in \Mats_n(\K)$ and $r \geq 2$.
Assume that $\rk(M+t(E_{i,j}+E_{j,i})) \leq r$ for all $t \in \K$.
Then, if we denote by $M_{i,j}$ the submatrix of $M$ obtained by deleting the $i$-th and $j$-th rows and columns,
we find $\rk M_{i,j} \leq r-2$.
\end{cor}

Our last basic lemma, which deals with symmetric matrices over fields of characteristic $2$, is somewhat more surprising:

\begin{lemme}\label{symmetriccornerlemma3}
Assume that $\K$ has characteristic $2$. Let $r$ be a positive integer such that $r<n$.
Let $A \in \Mats_n(\K)$, and $C \in \K^{n-1}$. Let us split
$$A=\begin{bmatrix}
P & C_0 \\
C_0^T & a_0
\end{bmatrix} \quad \text{with $P \in \Mats_{n-1}(\K)$, $C_0 \in \K^{n-1}$ and $a_0 \in \K$,}$$
and let us set
$$N:=\begin{bmatrix}
[0]_{(n-1) \times (n-1)} & C \\
C^T & 1
\end{bmatrix}.$$
Assume that $\rk(A+tN) \leq r$ for all $t \in \K$. Then, $\rk(P+CC^T) \leq r-1$.
\end{lemme}

\begin{proof}
Assume on the contrary that $s:=\rk(P+CC^T)$ satisfies $s \geq r$. \\
If we choose $Q \in \GL_{n-1}(\K)$ and set $\widetilde{Q}:=Q \oplus I_1$, then we note that
$$\widetilde{Q} A\widetilde{Q}^T=\begin{bmatrix}
QPQ^T & QC_0 \\
(QC_0)^T & a_0
\end{bmatrix} \quad,  \quad
\widetilde{Q} N\widetilde{Q}^T=\begin{bmatrix}
[0]_{(n-1) \times (n-1)} & QC \\
(QC)^T & 1
\end{bmatrix}$$
and $Q(P+CC^T)Q^T=(QPQ^T)+(QC)(QC)^T$.
Thus, working as in the above proofs, we lose no generality in assuming that
$$P+CC^T=\begin{bmatrix}
B & [0]_{s \times (n-1-s)} \\
[0]_{(n-1-s) \times s} & [0]_{(n-1-s) \times (n-1-s)}
\end{bmatrix} \quad \text{where $B \in \Mats_s(\K) \cap \GL_s(\K)$,}$$
in which case we write
$$P=\begin{bmatrix}
P_1 & [?]_{s \times (n-1-s)} \\
[?]_{(n-1-s) \times s} & [?]_{(n-1-s) \times (n-1-s)}
\end{bmatrix}$$
along the same pattern. Writing
$C=\begin{bmatrix}
C_1 \\
C_2
\end{bmatrix}$ and $C_0=\begin{bmatrix}
C_{0,1} \\
C_{0,2}
\end{bmatrix}$
with $C_1$ and $C_{0,1}$ in $\K^s$ and $C_2$ and $C_{0,2}$ in $\K^{n-1-s}$, we see that
$$B=P_1+C_1 C_1^T.$$
Set $$A':=\begin{bmatrix}
P_1 & C_{0,1} \\
C_{0,1}^T & a_0
\end{bmatrix} \quad \text{and} \quad N':=\begin{bmatrix}
[0]_{s \times s} & C_1 \\
C_1^T & 1
\end{bmatrix},$$
both of which are matrices of $\Mats_{s+1}(\K)$.
Let $t \in \K$:
as $s+1>r$ and $A'+tN'$ is a submatrix of $A+tN$ we find that $A'+tN'$ is singular.
Applying point (b) of Lemma \ref{symmetricdeterminantlemma} yields
$$\det P_1=C_1^T (P_1)^\ad C_1.$$
Yet, a classical formula for rank $1$ perturbations of the determinant states that
$$\det\bigl(P_1+C_1 C_1^T\bigr)=\det (P_1)+C_1^T (P_1)^\ad C_1=0.$$
This yields $\det B=0$, contradicting our assumption that $B$ is non-singular.
Therefore, $\rk(P+CC^T) <r$, as claimed.
\end{proof}

\subsection{On the rank $2$ matrices in the translation vector space of a bounded rank subspace}\label{rank2section}

\begin{lemme}\label{keylemmasymmetriccarnot2}
Assume that $\K$ has characteristic not $2$.
Let $s$ be a non-negative integer such that $2s \leq n$.
Let $\calS$ be an affine subspace of $\Mats_n(\K)$ whose translation vector space we denote by $S$.
Assume that $\dim S_H \geq s$ for every linear hyperplane $H$ of $\K^n$.
Then, $\urk \calS \geq 2s$.
\end{lemme}

\begin{proof}
Set $r:=\urk \calS$ and assume that $r<2s$.
Let $A \in \calS$ be of rank $r$. Replacing $\calS$ with a congruent subspace, we see that no generality is lost in assuming that
$$A=\begin{bmatrix}
P & [0]_{r \times (n-r)} \\
[0]_{(n-r) \times r} & [0]_{(n-r) \times (n-r)}
\end{bmatrix}$$
for some $P \in \GL_r(\K) \cap \Mats_r(\K)$.
Let us consider the hyperplane $H$ of $\K^n$ defined by the equation $x_{r+1}=0$ in the canonical basis.
For any matrix $N \in \calS_H$, let us write
$$N=\begin{bmatrix}
[0]_{r \times r} & C(N) \\
C(N)^T & D(N)
\end{bmatrix} \quad \text{with $C(N) \in \Mat_{r,n-r}(\K)$ and $D(N) \in \Mats_{n-r}(\K)$,}$$
and let us further split
$$C(N)=\begin{bmatrix}
C_1(N) & [0]_{r \times (n-r-1)}
\end{bmatrix} \quad \text{with $C_1(N) \in \K^r$.}$$
Note that
$$A+N=\begin{bmatrix}
P & C(N) \\
C(N)^T & D(N)
\end{bmatrix},$$
whereas $\rk C(N) \leq 1$. As $\rk(A+N) \leq r$, we deduce from Lemma \ref{schurcomplementlemma} that
$D(N)=C(N)^T P^{-1} C(N)$.
Applying this to $tN$ we deduce that $\forall t \in \K, \; tD(N)=t^2\,C(N)^T P^{-1} C(N)$, and since $\K$ has more than $2$ elements this yields
$$D(N)=0 \quad \text{and} \quad C(N)^T P^{-1} C(N)=0.$$
It follows from the first identity that $\dim C_1(S_H)=\dim S_H \geq s$,
and the second one yields that $C_1(S_H)$ is a totally isotropic subspace of $\K^r$
for the regular quadratic form $X \mapsto X^T P^{-1} X$, leading to $\dim C_1(S_H) \leq \frac{r}{2}<s$.
This is a contradiction.

We conclude that $r \geq 2s$, as claimed.
\end{proof}

\begin{lemme}\label{keylemmaalternating}
Let $s$ be a non-negative integer such that $2s \leq n$.
Let $\calS$ be an affine subspace of $\Mata_n(\K)$ whose translation vector space we denote by $S$.
Assume that $\dim S_H \geq s$ for every linear hyperplane $H$ of $\K^n$.
Then, $\urk \calS \geq 2s$.
\end{lemme}

\begin{proof}
Set $r:=\urk \calS$ and assume that $r<2s$. Note that $r=2s'$ for some $s' \in \lcro 0,s-1\rcro$.

Let $A \in \calS$ be of rank $r$. Replacing $\calS$ with a congruent subspace, we see that no generality is lost in assuming that
$$A=\begin{bmatrix}
P & [0]_{r \times (n-r)} \\
[0]_{(n-r) \times r} & [0]_{(n-r) \times (n-r)}
\end{bmatrix}$$
for some $P \in \GL_r(\K) \cap \Mata_r(\K)$.
Let $H$ be an arbitrary linear hyperplane of $\K^n$ that includes $\K^r \times \{0\}$.
For any matrix $N \in S_H$, let us write
$$N=\begin{bmatrix}
[0]_{r \times r} & C(N) \\
-C(N)^T & D(N)
\end{bmatrix} \quad \text{with $C(N) \in \Mat_{r,n-r}(\K)$ and $D(N) \in \Mata_{n-r}(\K)$.}$$
Applying Lemma \ref{schurcomplementlemma}, we find $D(N)=-C(N)^T P^{-1} C(N)$.
As $\rk C(N) \leq 1$, this yields $\rk D(N) \leq 1$. Since $D(N)$ is alternating, we conclude that $D(N)=0$.

Noting that $n-r \geq 2$, we can consider the hyperplanes $H_1$ of $H_2$ of $\K^n$ defined, respectively, by the equations $x_{r+1}=0$ and $x_{r+2}=0$
in the standard basis. Let $(N_1,N_2) \in S_{H_1} \times S_{H_2}$. Then, we can write
$$C(N_1)=\begin{bmatrix}
C_1(N_1) & [0]_{r \times 1} & [0]_{r \times (n-r-2)}
\end{bmatrix} \quad \text{and} \quad
C(N_2)=\begin{bmatrix}
[0]_{r \times 1} & C_2(N_2) & [0]_{r \times (n-r-2)}
\end{bmatrix}$$
with $C_1(N_1)$ and $C_2(N_2)$ in $\K^r$.

As $\rk(A+N_1+N_2)\leq r$ and $D(N_1)+D(N_2)=0$, we deduce once more from Lemma \ref{schurcomplementlemma} that
$$\begin{bmatrix}
C_1(N_1)^T P^{-1} C_1(N_1) & C_1(N_1)^T P^{-1}C_2(N_2) \\
C_2(N_2)^T P^{-1} C_1(N_1) & C_2(N_2)^T P^{-1}C_2(N_2)
\end{bmatrix}=0,$$
and in particular
$$C_1(N_1)^T P^{-1} C_2(N_2)=0.$$
It follows that the linear subspaces $C_1(S_{H_1})$ and $C_2(S_{H_2})$
are orthogonal for the non-degenerate alternating bilinear form $(X,Y) \mapsto X^T P^{-1} Y$ on $\K^r$.
Therefore, $\dim C_1(S_{H_1})+\dim C_2(S_{H_2}) \leq r$.
On the other hand, $\dim C_1(S_{H_1})=\dim S_{H_1} \geq s$ and $\dim C_2(S_{H_2})=\dim S_{H_2} \geq s$,
contradicting $2s>r$. Therefore, $r \geq 2s$, as claimed.
\end{proof}

\begin{lemme}\label{keylemmasymmetriccar2}
Assume that $\K$ has characteristic $2$.
Let $s$ be a non-negative integer such that $2s \leq n$.
Let $\calS$ be an affine subspace of $\Mats_n(\K)$ whose translation vector space we denote by $S$. \\
Assume that $\dim S_H \geq s$ for every linear hyperplane $H$ of $\K^n$.
\begin{enumerate}[(a)]
\item If $n>2s$ then $\urk \calS \geq 2s$.
\item If $n=2s$ then $\urk \calS \geq n-1$.
\end{enumerate}
\end{lemme}

\begin{proof}
Set $r:=\urk \calS$.
Assume that $r<2s$ and $r \leq n-2$.
Let $A \in \calS$ be of rank $r$. Replacing $\calS$ with a congruent subspace, we see that no generality is lost in assuming that
$$A=\begin{bmatrix}
P & [0]_{r \times (n-r)} \\
[0]_{(n-r) \times r} & [0]_{(n-r) \times (n-r)}
\end{bmatrix} \quad \text{for some $P \in \Mats_r(\K) \cap \GL_r(\K)$.}$$
Let $H$ be an arbitrary linear hyperplane of $\K^n$ that includes $\K^r \times \{0\}$.
For any matrix $N \in S_H$, let us write
$$N=\begin{bmatrix}
[0]_{r \times r} & C(N) \\
C(N)^T & D(N)
\end{bmatrix} \quad \text{with $C(N) \in \Mat_{r,n-r}(\K)$ and $D(N) \in \Mats_{n-r}(\K)$.}$$
Applying Lemma \ref{schurcomplementlemma}, we find
$$D(N)=C(N)^T \,P^{-1}\, C(N).$$
However, as $\rk C(N)=1$ we deduce that $\rk D(N) \leq 1$, and if $C(N)=0$ then $D(N)=0$.
In particular we have $\dim C(S_H)=\dim S_H \geq s$.
Then, we proceed as in the proof of Lemma \ref{keylemmaalternating}:
we consider the hyperplanes $H_1$ of $H_2$ of $\K^n$ defined, respectively, by the equations $x_{r+1}=0$ and $x_{r+2}=0$ in the standard basis.
Let $(N_1,N_2) \in S_{H_1} \times S_{H_2}$. Then, we write
$$C(N_1)=\begin{bmatrix}
C_1(N_1) & [0]_{r \times 1} & [0]_{r \times (n-r-2)}
\end{bmatrix} \quad ; \quad
C(N_2)=\begin{bmatrix}
[0]_{r \times 1} & C_2(N_2) & [0]_{r \times (n-r-2)}
\end{bmatrix}$$
with $C_1(N)$ and $C_2(N_2)$ in $\K^r$, and we further write
$$D(N_1)=\begin{bmatrix}
? & 0 & [0]_{1 \times (n-r-2)} \\
0 & 0 & [0]_{1 \times (n-r-2)} \\
[0]_{(n-r-2) \times 1} & [0]_{(n-r-2) \times 1} & [0]_{(n-r-2) \times (n-r-2)}
\end{bmatrix}$$
and
$$D(N_2)=\begin{bmatrix}
0 & 0 & [0]_{1 \times (n-r-2)} \\
0 & ? & [0]_{1 \times (n-r-2)} \\
[0]_{(n-r-2) \times 1} & [0]_{(n-r-2) \times 1} & [0]_{(n-r-2) \times (n-r-2)}
\end{bmatrix}.$$
Lemma \ref{schurcomplementlemma} applied to $A+N_1+N_2$ yields
$$D(N_1)+D(N_2)=\bigl(C(N_1)^T+C(N_2)^T\bigr) P^{-1} \bigl(C(N_1)+C(N_2)\bigr).$$
Evaluating both sides of this identity at the $(1,2)$-spot, we deduce that
$$C_1(N_1)^T P^{-1} C_2(N_2)=0.$$
Therefore, $C_1(S_{H_1})$ and $C_2(S_{H_2})$ are orthogonal subspaces for the non-degenerate symmetric bilinear form
$(X,Y) \mapsto X^T P^{-1} Y$ on $\K^r$, which yields
$$\dim C_1(S_{H_1})+\dim C_2(S_{H_2}) \leq r,$$
contradicting the assumption that $2s>r$.

It follows that either $r \geq 2s$ or $r \geq n-1$, yielding the claimed results.
\end{proof}

The above three lemmas will systematically be used in the form of their contraposition.
Hence, with an upper-bound on the upper-rank of $\calS$, we shall find a linear hyperplane $H$ of $\K^n$
for which the dimension of $S_H$ is small. There is only one situation in which this fails: it is the case when $n$
is odd, $\K$ has characteristic $2$, the upper-rank is $n-1$ and we are dealing with a subspace of symmetric matrices.
Then, we have a simple example which shows that it is possible that no linear hyperplane $H$ of $\K^n$ satisfies
$\dim S_H \leq \frac{n-1}{2}\cdot$ Assume indeed that $\K$ has characteristic $2$ and that $n$ is odd, and consider
the subspace $\calS=\Mata_n(\K)$ of $\Mats_n(\K)$. It has upper-rank $n-1$, yet one sees that
$\dim S_H=n-1$ for every linear hyperplane $H$ of $\K^n$. To circumvent that problem, we shall prove a result that deals with that special case:

\begin{lemme}\label{lastkeylemma}
Assume that $\K$ has characteristic $2$ and that $n \geq 2$.
Let $\calS$ be an affine subspace of $\Mats_n(\K)$ in which every matrix is singular.
Denote by $S$ its translation vector space.  \\
Then, either $n$ is odd and $\calS=\Mata_n(\K)$, or
there exists a linear hyperplane $H$ of $\K^n$ such that $S_H$ does not include $\bigl(\Mata_n(\K)\bigr)_{H.}$
\end{lemme}

To better grasp the meaning of the conclusion, let us consider the special case of the hyperplane $H=\K^{n-1} \times \{0\}$.
Then, the matrices of $S_H$ may be written as
$$M=\begin{bmatrix}
[0]_{(n-1) \times (n-1)} & C(M) \\
C(M)^T & a(M)
\end{bmatrix} \quad \text{with $C(M)\in \K^{n-1}$ and $a(M) \in \K$.}$$
For $S_H$ not to include $\bigl(\Mata_n(\K)\bigr)_H$, it is necessary and sufficient
that either $a(S_H) \neq \{0\}$ and $\dim S_H \leq n-1$, or $a(S_H)=\{0\}$ and $\dim S_H \leq n-2$.

\begin{proof}
Let us assume that $S_H$ includes $(\Mata_n(\K))_H$ for every linear hyperplane $H$ of $\K^n$.
In particular, for every $i \in \lcro 1,n\rcro$, this holds for the hyperplane defined in the standard basis by the equation $x_i=0$.
It follows that $S$ includes $\Mata_n(\K)$. Assume that $\calS$ contains a non-zero diagonal matrix $D$.
With no loss of generality, we can assume that
$D=\Diag(0,\dots,0,a_{p+1},\dots,a_n)$ for some $p \in \lcro 1,n-1\rcro$, where $a_{p+1},\dots,a_n$ are non-zero scalars.
For all $i \in \lcro 1,\lfloor \frac{p+1}{2}\rfloor\rcro$, we know that $S$ contains the alternating matrix
$A_i=E_{2i-1,2i}-E_{2i,2i-1}$.
One then checks that the matrix $D+\underset{i=1}{\overset{\lfloor \frac{p+1}{2}\rfloor}{\sum}} A_i$
is non-singular, contradicting our assumptions on $\calS$. Hence, $\calS$ contains no non-zero diagonal matrix.

Now, given a matrix $M \in \calS$, the alternating matrix $A \in \Mata_n(\K)$ with the same off-diagonal elements as $M$
belongs to $S$, and hence the diagonal matrix $D:=M-A$ belongs to $\calS$. It follows that $M=A \in \Mata_n(\K)$.
Since $A \in S$ we deduce that $0 \in \calS$ and hence $\calS=S$: then, it follows from the first part of the proof that
$\Mata_n(\K) \subset \calS$, and from the second one that $\calS \subset \Mata_n(\K)$.
Hence, $\calS=\Mata_n(\K)$.
If $n$ is even then $\Mata_n(\K)$ contains the non-singular matrix
$\begin{bmatrix}
[0]_{s \times s} & I_s \\
-I_s & [0]_{s \times s}
\end{bmatrix}$ where $s:=\frac{n}{2}\cdot$ We conclude that $n$ is odd.
\end{proof}

\subsection{Range-compatible maps on full spaces of symmetric or alternating matrices}\label{RCsection}

We recall the following notion from \cite{dSPRC1}.

\begin{Def}
Let $U$ and $V$ be vector spaces, and $\calS$ be a linear subspace of $\calL(U,V)$, the space of all linear mappings
from $U$ to $V$. A map $F : \calS \rightarrow V$ is called \textbf{range-compatible} when
$$\forall s \in \calS, \; F(s) \in \im(s).$$
It is called \textbf{local} when $F : s \mapsto s(x)$ for some vector $x \in U$.
\end{Def}

Of course, we can interpret any linear subspace of $\Mat_p(\K)$ as a subspace of $\calL(\K^p,\K^p)$
by using the canonical basis, and hence we have a notion of range-compatibility for maps from such a subspace to $\K^p$.

Range-compatible additive maps on large spaces of rectangular matrices have been extensively studied in the recent \cite{dSPRC1,
dSPRC2,dSPRCsym}. In this work, we shall need a precise understanding of the range-compatible maps on
the special spaces $\Mats_n(\K)$ and $\Mata_n(\K)$.
Let us recall the known results:

\begin{theo}[Theorem 1.7 of \cite{dSPRC1}]\label{symmetricRC}
Let $p$ be a non-negative integer.
\begin{enumerate}[(a)]
\item If $\# \K>2$ then every range-compatible linear map on $\Mats_p(\K)$ is local.
\item If $\# \K=2$ then every range-compatible linear map on $\Mats_p(\K)$ is local or equals the sum of a local map
with $M \mapsto \Delta(M)$.
\end{enumerate}
\end{theo}

Conversely, it can be checked that $M \mapsto \Delta(M)$ is range-compatible on $\Mats_p(\F_2)$
and that it is non-local if $p>1$.

\begin{theo}[Theorem 1.7 of \cite{dSPRCsym}]\label{alternatingRC}
Let $p$ be a non-negative integer. Then, every range-compatible linear map on $\Mata_p(\K)$ is local.
\end{theo}

\subsection{On affine spaces of matrices with rank at most $1$}\label{rankatmost1section}

\begin{prop}\label{rankatmost1lemma}
Let $\calS$ be a $1$-dimensional affine subspace of $\Mats_n(\K)$ in which every matrix has rank at most $1$.
Then:
\begin{itemize}
\item Either $\calS$ is congruent to $\widetilde{\Mats_1(\K)}^{(n)}$;
\item Or $n \geq 2$, $\# \K=2$ and $\calS$ is congruent to $\widetilde{Z_2(\K)}^{(n)}$.
\end{itemize}
\end{prop}

\begin{proof}
It is obvious that the first option holds true if $n=1$ or if $0 \in \calS$.
Assume now that $n \geq 2$ and that $0 \not\in \calS$.
We can choose two rank $1$ matrices $A$ and $B$ in $\calS$, so that $\calS$ is the line going through $A$ and $B$.
Note that every matrix of $\calS$ has its range included in $\im A+\im B$.

If $\im A=\im B$ then $\calS$ is congruent to a $1$-dimensional affine subspace of $\widetilde{\Mats_1(\K)}^{(n)}$,
and hence $0 \in \calS$, which contradicts our assumptions.

Assume now that $\im A \neq \im B$. Then, $\calS$ is congruent to a $1$-dimensional affine subspace of
$\widetilde{\Mats_2(\K)}^{(n)}$ and hence no generality is lost in assuming that $n=2$.
Then, as we can replace $A$ and $B$ with $PAP^T$ and $PBP^T$ for a well-chosen matrix $P \in \GL_2(\K)$,
we lose no generality in assuming that $A=\alpha\,E_{1,1}$ and $B=\beta\,E_{2,2}$ for some pair $(\alpha,\beta)\in (\K \setminus \{0\})^2$.
Then, for all $t \in \K$, the matrix $\begin{bmatrix}
\alpha\, t & 0 \\
0 & \beta\,(1-t)
\end{bmatrix}$ must have rank at most $1$, and hence $\# \K=2$ and $\calS=\{E_{1,1},E_{2,2}\}$.
Then, with $P:=\begin{bmatrix}
1 & 1 \\
0 & 1
\end{bmatrix}$, one checks that $P^T \calS P=Z_2(\K)$.
\end{proof}

\subsection{A corollary to Flanders's theorem}\label{Flanderscorsection}

\begin{cor}\label{Flanderscor}
Let $n$ and $p$ be non-negative integers such that $n \geq p \geq 2$. Let $\calV$
be an affine subspace of $\Mat_{n,p}(\K)$ with $\urk \calV<p$. \\
Assume that, for all $N \in \Mat_{n,p-1}(\K)$, the space $\calV$ contains a matrix of the form $\begin{bmatrix}
N & [?]_{n \times 1}
\end{bmatrix}$.
If $n>p$ or $p >2$ or $\# \K>2$ or $\calV$ contains the zero matrix, then
there exists a vector $Y \in \K^{p-1}$ such that
$$\calV=\Bigl\{\begin{bmatrix}
N & NY
\end{bmatrix} \mid N \in \Mat_{n,p-1}(\K)\Bigr\}.$$
\end{cor}

\begin{proof}
By Flanders's theorem, we know that $\dim \calV \leq n(p-1)$.
As $\calV$ contains a matrix of the form $\begin{bmatrix}
N & [?]_{n \times 1}
\end{bmatrix}$ for all $N \in \Mat_{n,p-1}(\K)$, it follows that $\dim \calV=n(p-1)$ and that there is an affine map
$C : \Mat_{n,p-1}(\K) \rightarrow \K^n$ such that
$$\calV=\biggl\{\begin{bmatrix}
N & C(N)
\end{bmatrix} \mid N \in \Mat_{n,p-1}(\K) \biggr\}.$$
If $n>p$ or $p>2$ or $\# \K>2$ or $\calV$ contains the zero matrix, then the second statement in Flanders's theorem shows that:
\begin{itemize}
\item Either there exists a non-zero vector $X \in \K^p$ such that $MX=0$ for all $M \in \calV$;
\item Or $n=p$ and there exists a non-zero vector $Y \in \K^n$ such that $Y^T M=0$ for all $M \in \calV$.
\end{itemize}
The second case is actually impossible because it would yield a non-zero vector $Y \in \K^n$ such that $Y^T N=0$ for all $N \in \Mat_{n,p-1}(\K)$.
Thus, we have a non-zero vector $X \in \K^p$ such that $MX=0$ for all $M \in \calV$.
The vector $X$ cannot belong to $\K^{p-1} \times \{0\}$ for this would yield a non-zero vector $X' \in \K^{p-1}$
such that $NX'=0$ for all $N \in \Mat_{n,p-1}(\K)$. Thus, replacing $X$ with a non-zero collinear vector if necessary, we can
assume that $X=\begin{bmatrix}
Y \\
-1
\end{bmatrix}$ for some $Y \in \K^{p-1}$, and it follows that $C(N)=NY$ for all $N \in \Mat_{n,p-1}(\K)$, which yields the claimed statement.
\end{proof}

\section{The maximal dimension for spaces of alternating matrices}\label{section3}

In this section, we prove the first statement of Theorem \ref{maintheoalternate}.
To do so, we perform an induction on $n$ and $r$. The case when $n \leq 2$ is obvious.
Assume now that $n \geq 3$. Let $s$ be a non-negative integer such that $2s<n$. Set $r:=2s$.
Let $\calS$ be an affine subspace of $\Mata_n(\K)$ whose translation vector space will be denoted by $S$.
Assume that $\urk \calS \leq 2s$.

If $n=2s+1$, then we simply write
$$\dim \calS \leq \dim \Mata_n(\K)=\dbinom{n}{2}=a^{(1)}_{n,2s.}$$
If $s=0$, it is obvious that $\dim \calS=0=a^{(1)}_{n,2s}$.

In the rest of the proof, we assume that $0<2s \leq n-2$. Then,
by Lemma \ref{keylemmaalternating}, we can find a linear hyperplane $H$ of $\K^n$ such that $\dim S_H \leq s$.
Replacing $\calS$ with a congruent space if necessary, we can assume that $H=\K^{n-1} \times \{0\}$.
From there, we split the discussion into two cases.

\subsection{Case 1: $S_H=\{0\}$.}\label{case1dimensionalternating}

In other words, $S$ contains no non-zero matrix of the form $\begin{bmatrix}
[0]_{(n-1) \times (n-1)} & [?]_{(n-1) \times 1} \\
[?]_{1 \times (n-1)} & 0
\end{bmatrix}$.

Then, we split every matrix $M$ of $\calS$ up as
$$M=\begin{bmatrix}
P(M) & [?]_{(n-1) \times 1} \\
[?]_{1 \times (n-1)} & 0
\end{bmatrix} \quad \text{with $P(M) \in \Mata_{n-1}(\K)$.}$$
Note that $P(\calS)$ is an affine subspace of $\Mata_{n-1}(\K)$ such that
$\urk P(\calS) \leq \urk \calS \leq 2s$.
Note that $2s<n-1$. Hence, by induction
$$\dim P(\calS) \leq \dbinom{2s+1}{2} \quad \text{or} \quad \dim P(\calS) \leq \dbinom{s}{2}+s(n-1-s).$$
On the other hand, as $S_H=\{0\}$ we find
$$\dim \calS=\dim P(\calS).$$
As $\dbinom{s}{2}+s(n-1-s) \leq \dbinom{s}{2}+s(n-s)$, we conclude that
$$\dim \calS \leq \dbinom{2s+1}{2} \quad \text{or} \quad \dim \calS \leq \dbinom{s}{2}+s(n-s),$$
that is $\dim \calS \leq \max\bigl(a_{n,r}^{(1)},a_{n,r}^{(2)}\bigr)$.

\subsection{Case 2: $S_H \neq \{0\}$.}\label{case2dimensionalternating}

Replacing $\calS$ with $(P \oplus I_1) \calS (P \oplus I_1)^T$ for a well-chosen
invertible matrix $P \in \GL_{n-1}(\K)$, we see that no generality is lost in assuming that
$S$ contains the matrix $N=E_{1,n}-E_{n,1}$. Let us write any matrix $M \in \calS$ as
$$M=\begin{bmatrix}
0 & [?]_{1 \times (n-2)} & ? \\
[?]_{(n-2) \times 1} & K(M) & [?]_{(n-2) \times 1} \\
? & [?]_{1 \times (n-2)} & 0
\end{bmatrix} \quad \text{with $K(M) \in \Mata_{n-2}(\K)$.}$$
Note that $K(\calS)$ is an affine subspace of $\Mata_{n-2}(\K)$ and that the rank theorem yields
$$\dim \calS \leq \dim K(\calS)+(n-2)+\dim S_H \leq \dim K(\calS)+(n-2)+s.$$
On the other hand, for all $M \in \calS$ and all $t \in \K$, the matrix
$M+t N$ belongs to $\calS$, and hence Lemma \ref{alternatingcornerlemma} yields that
$$\urk K(\calS) \leq 2s-2.$$
By induction, we deduce that
$$\dim K(\calS) \leq \max\bigl(a^{(1)}_{n-2,r-2},a^{(2)}_{n-2,r-2}\bigr).$$
One checks that
$$a^{(2)}_{n-2,r-2}+(n-2)+s=a^{(2)}_{n,r.}$$
Hence, if $a^{(2)}_{n-2,r-2} \geq a^{(1)}_{n-2,r-2}$ then
$$\dim \calS \leq a^{(2)}_{n,r.}$$
Assume now that $a^{(1)}_{n-2,r-2}>a^{(2)}_{n-2,r-2}$ which, by Remark \ref{dimensionremarkalternating},
shows that $5(s-1)\geq 2(n-2)-3$, that is $5s\geq 2n-2$.
Then,
$$\dim \calS \leq \dbinom{2s-1}{2}+(n-2)+s,$$
whereas
\begin{align*}
\dbinom{2s+1}{2}-\Biggl( \dbinom{2s-1}{2}+(n-2)+s\Biggr)& =2s+(2s-1)-s-n+2 \\
& =3s-n+1 \\
& \geq \frac{n-1}{5}>0.
\end{align*}
Hence,
$$\dim \calS < \dbinom{2s+1}{2}=a^{(1)}_{n,r} \leq \max\bigl(a^{(1)}_{n,r},a^{(2)}_{n,r}\bigr).$$
Therefore, the desired conclusion holds in either case, which finishes our proof.

\subsection{Some corollaries}\label{corollalternatingsection}

The following result is an obvious consequence of the inequality statement in Theorem \ref{maintheoalternate}:

\begin{cor}\label{generationcor1alternating}
Let $n$ and $r$ be non-negative integers. Assume that $r$ is even and that $r<n$.
Let $\calS$ be an affine subspace of $\Mata_n(\K)$ such that
$$\dim \calS \geq 2+\max\bigl(a^{(1)}_{n,r},a^{(2)}_{n,r}\bigr).$$
Then, the affine space $\calS$ is generated by its matrices with rank greater than $r$.
\end{cor}

\begin{proof}
If the contrary held true then some hyperplane $\calT$ of $\calS$
would contain all the matrices of $\calS$ with rank greater than $r$; choosing a different parallel hyperplane
$\calT'$, we would see that $\urk \calT' \leq r$, and by the inequality statement in Theorem \ref{maintheoalternate} this would lead to
$$\dim \calT' \leq \max\bigl(a^{(1)}_{n,r},a^{(2)}_{n,r}\bigr),$$
contradicting the fact that $\dim \calT' = \dim \calS -1$.
\end{proof}

In particular, we obtain the following known result as a corollary.

\begin{cor}\label{generationcor2alternating}
Let $r$ be a positive even integer.
Then:
\begin{enumerate}[(a)]
\item The affine space $\Mata_{r+1}(\K)$ is generated by its rank $r$ matrices.
\item The affine space $\Mata_r(\K)$ is generated by its rank $r$ matrices unless $r=2$ and $\# \K=2$.
\end{enumerate}
\end{cor}

When $r=2$, this result is not derived from Corollary \ref{generationcor1alternating}.
Rather, we simply remark that there are several rank $2$ matrices in the line $\Mata_2(\K)$ if $\# \K>2$.

\section{The maximal dimension for spaces of symmetric matrices: the characteristic not $2$ case}\label{section4}

In this section, we assume that the characteristic of $\K$ differs from $2$,
and we prove the first statement of Theorem \ref{maintheosymmetric} in that case.
We perform an induction on $n$ and $r$.
Let $n$ and $r$ be non-negative integers such that $r<n$.
Let $\calS$ be an affine subspace of $\Mats_n(\K)$ whose translation vector space we denote by $S$.
Assume that $\urk \calS \leq r$. If $r=0$ then $\calS=\{0\}$ and hence $\dim \calS \leq s_{n,r}^{(1)}$.

In the rest of the proof, we assume that $r>0$.
By Lemma \ref{keylemmasymmetriccarnot2}, there exists a linear hyperplane $H$ of $\K^n$ such that
$$\dim S_H \leq \left\lfloor \frac{r}{2}\right\rfloor.$$
Without loss of generality, we can assume that $H=\K^{n-1} \times \{0\}$.
From there, we split the discussion into several subcases.

\subsection{Case 1: $S_H=\{0\}$.}\label{section4.1}

Then, by working as in Section \ref{case1dimensionalternating} we find by induction that
$$\dim \calS \leq \max\bigl(s^{(1)}_{n-1,r},s^{(2)}_{n-1,r}\bigr)
\leq \max\bigl(s^{(1)}_{n,r},s^{(2)}_{n,r}\bigr).$$

\subsection{Case 2: $S_H$ contains a rank $2$ matrix.}\label{section4.2}

As in Section \ref{case2dimensionalternating}, we can use a well-chosen congruence transformation to reduce the situation to the one where
$S$ contains
$$N=\begin{bmatrix}
0 & [0]_{1 \times (n-2)} & 1 \\
[0]_{(n-2) \times 1} & [0]_{(n-2) \times (n-2)} & [0]_{(n-2) \times 1} \\
1 & [0]_{1 \times (n-2)} & a
\end{bmatrix} \quad \text{for some $a \in \K$.}$$
Then, we split every matrix $M$ of $\calS$ up as
$$M=\begin{bmatrix}
? & [?]_{1 \times (n-2)} & ? \\
[?]_{(n-2) \times 1} & K(M) & [?]_{(n-2) \times 1} \\
? & [?]_{1 \times (n-2)} & ?
\end{bmatrix} \quad \text{with $K(M) \in \Mats_{n-2}(\K)$.}$$
Note that $K(\calS)$ is an affine subspace of $\Mats_{n-2}(\K)$ and that the rank theorem yields
$$\dim \calS \leq \dim K(\calS)+(n-1)+\dim S_H \leq \dim K(\calS)+(n-1)+\left\lfloor \frac{r}{2}\right\rfloor.$$
On the other hand, for all $M \in \calS$ and all $t \in \K$, the matrix
$M+t N$ belongs to $\calS$, and hence Lemma \ref{symmetriccornerlemma1} yields that
$$\urk K(\calS) \leq r-2.$$
In particular, $r \geq 2$. By induction, we deduce that
$$\dim K(\calS) \leq s^{(1)}_{n-2,r-2} \quad \text{or} \quad \dim K(\calS) \leq s^{(2)}_{n-2,r-2.}$$
Note in any case that
$$s^{(2)}_{n,r}=s^{(2)}_{n-2,r-2}+(n-1)+\left\lfloor \frac{r}{2}\right\rfloor.$$
Therefore, if $\dim K(\calS) \leq s^{(2)}_{n-2,r-2}$ then $\dim \calS \leq s^{(2)}_{n,r}$.
Assume now that $s^{(1)}_{n-2,r-2}> s^{(2)}_{n-2,r-2}$, so that $r-2>1$.
Then, we split the discussion into two subcases, whether $r$ is even or odd.
\begin{itemize}
\item Assume that $r=2s$ for some integer $s \geq 1$. By Remark \ref{dimensionremarksymmetric}, we
find that $5(s-1)\geq 2(n-2)-1$, that is $5s \geq 2n$.
On the other hand,
$$\dim \calS \leq \dbinom{2s-1}{2}+(n-1)+s,$$
whereas
\begin{align*}
\dbinom{2s+1}{2}-\dbinom{2s-1}{2}-(n-1)-s & =2s+2s-1-n+1-s \\
& =3s-n>0,
\end{align*}
which yields
$$\dim \calS < \dbinom{2s+1}{2}.$$

\item Assume that $r=2s+1$ for some integer $s\geq 1$. By Remark \ref{dimensionremarksymmetric}, we
find that $5(s-1)\geq 2(n-2)-5$, that is $5(s+1)\geq 2(n+1)-1$.
On the other hand,
$$\dim \calS \leq \dbinom{2s}{2}+(n-1)+s,$$
whereas
\begin{align*}
\dbinom{2s+2}{2}-\dbinom{2s}{2}-(n-1)-s & =2s+1+2s-n+1-s \\
& =3(s+1)-(n+1) \\
& \geq \frac{n-2}{5}>0,
\end{align*}
which yields $\dim \calS < \dbinom{2s+2}{2}$.
\end{itemize}
In any case we have proved that $\dim \calS \leq \max\bigl(s^{(1)}_{n,r},s^{(2)}_{n,r}\bigr)$.

\subsection{Case 3: $S_H$ is non-zero and contains no rank $2$ matrix}\label{section4.3}

In particular, $\dim S_H=1$ and $S_H$ contains $E_{n,n}$.
Let us then split every matrix $M \in \calS$ up as
$$M=\begin{bmatrix}
P(M) & [?]_{(n-1) \times 1} \\
[?]_{1 \times (n-1)} & ?
\end{bmatrix} \quad \text{with $P(M) \in \Mats_{n-1}(\K)$.}$$
Using Lemma \ref{symmetriccornerlemma2}, we find that the affine subspace
$P(\calS)$ of $\Mats_{n-1}(\K)$ satisfies
$$\urk P(\calS) \leq r-1.$$
Moreover, the rank theorem yields
$$\dim \calS=\dim P(\calS)+1.$$
By induction, we have
$$\dim P(\calS) \leq s^{(1)}_{n-1,r-1} \quad \text{or} \quad \dim P(\calS) \leq s^{(2)}_{n-1,r-1.}$$
Yet, $s^{(1)}_{n,r}-s^{(1)}_{n-1,r-1}=r\geq 1$ and hence the first outcome would yield
$$\dim P(\calS)\leq s^{(1)}_{n,r.}$$
Moreover, one checks that
$$s^{(2)}_{n,r}-s^{(2)}_{n-1,r-1}=\begin{cases}
1+\frac{r-1}{2} & \text{if $r$ is odd} \\
n-1             & \text{if $r$ is even.}
\end{cases}$$
In any case, we see that $s^{(2)}_{n-1,r-1}+1 \leq s^{(2)}_{n,r}$.
Hence,
$$\dim P(\calS) \leq s^{(2)}_{n-1,r-1} \Rightarrow \dim \calS \leq s^{(2)}_{n,r.}$$
Thus, in any case we have proved that
$$\dim \calS\leq \max\bigl(s^{(1)}_{n,r}, s^{(2)}_{n,r}\bigr).$$
This completes our inductive proof of the inequality statement in Theorem \ref{maintheosymmetric}.

\section{The maximal dimension for spaces of symmetric matrices: the characteristic $2$ case}\label{section5}

In this section, we complete the proof of the inequality statement in Theorem
\ref{maintheosymmetric} by tackling the special case of fields of characteristic $2$.
Here, things are made somewhat more complex by the failure of the conclusion of Lemma \ref{symmetriccornerlemma1}
when $d \neq 0$ and $\# \K=2$, and by the fact that Lemma \ref{keylemmasymmetriccar2} does not yield any satisfying result when $n=2s+1$.

Assume that $\K$ has characteristic $2$.
Let $n$ and $r$ be non-negative integers such that $r<n$.
Let $\calS$ be an affine subspace of $\Mats_n(\K)$ whose translation vector space will be denoted by $S$.
Assume that $\urk \calS \leq r$. If $r=0$, we obviously have $\dim \calS=0 =\max\bigl(s_{n,r}^{(1)},s_{n,r}^{(2)}\bigr)$.
In the rest of the proof, we assume that $r>0$.
From there, we distinguish between two main cases, whether
$r<n-1$ or $r=n-1$. In the first case, we shall rely upon Lemma \ref{keylemmasymmetriccar2}, whereas
in the second one we will use Lemma \ref{lastkeylemma}.

\subsection{Case 1: $r<n-1$.}

Lemma \ref{keylemmasymmetriccar2} yields a linear hyperplane $H$ of $\K^n$ such that
$\dim S_H \leq \lfloor \frac{r}{2}\rfloor$.
Without loss of generality, we can assume that $H=\K^{n-1} \times \{0\}$.
Then, we distinguish between several subcases, according to the shape of $S_H$.

\subsubsection{Subcase 1.1: $S_H=\{0\}$.}

Then, we proceed as in Section \ref{section4.1} to obtain
$$\dim \calS \leq \max\left(s^{(1)}_{n-1,r},s^{(2)}_{n-1,r}\right) \leq \max\left(s^{(1)}_{n,r},s^{(2)}_{n,r}\right).$$

\subsubsection{Subcase 1.2: $S_H$ contains a non-zero alternating matrix.}

Then, without loss of generality we can assume that $S_H$ contains $E_{1,n}+E_{n,1}$.
Then, as in Section \ref{section4.1}, we use Lemma \ref{symmetriccornerlemma1} and the induction hypothesis to obtain
$$\dim \calS \leq (n-1)+\left\lfloor \frac{r}{2}\right\rfloor+\max\left(s^{(1)}_{n-2,r-2},s^{(2)}_{n-2,r-2}\right)
\leq \max\left(s^{(1)}_{n,r},s^{(2)}_{n,r}\right).$$

\subsubsection{Subcase 1.3: $S_H$ is non-zero and contains no alternating matrix.}\label{section5.1.3}

In particular, $\dim S_H=1$ and
$S$ contains $\begin{bmatrix}
[0]_{(n-1) \times (n-1)} & C \\
C^T & 1
\end{bmatrix}$ for some $C \in \K^{n-1}$.
Then, we split every matrix $M \in \calS$ up as
$$M=\begin{bmatrix}
P(M) & [?]_{(n-1) \times 1} \\
[?]_{1 \times (n-1)} & ?
\end{bmatrix} \quad \text{with $P(M) \in \Mats_{n-1}(\K)$.}$$
By Lemma \ref{keylemmasymmetriccar2},
the affine space $\calT:=CC^T+P(\calS)$ has upper-rank less than $r$.
By induction, we deduce that
$$\dim P(\calS)=\dim \calT \leq \max\left(s^{(1)}_{n-1,r-1},s^{(2)}_{n-1,r-1}\right),$$
and since $\dim S_H=1$ we use the rank theorem like in Section \ref{section4.3} to conclude that
$$\dim \calS \leq \max\left(s^{(1)}_{n,r},s^{(2)}_{n,r}\right).$$

\subsection{Case 2: $r=n-1$.}

In that case, we note that
$$\max\left(s^{(1)}_{n,r},s^{(2)}_{n,r}\right)=s^{(1)}_{n,n-1}=\dbinom{n}{2}.$$
If $\calS=\Mata_n(\K)$, then we readily have $\dim \calS=\dbinom{n}{2}$, which is the desired outcome. \\
In the rest of this section, we assume that $\calS \neq \Mata_n(\K)$.
By Lemma \ref{lastkeylemma}, we can then find a linear hyperplane $H$ of $\K^n$ such that
$S_H$ does not include $(\Mata_n(\K))_H$. In particular, $\dim S_H \leq n-1$.
Without loss of generality, we can assume that $H=\K^{n-1} \times \{0\}$.

Then, we split the discussion once more into several subcases.

\subsubsection{Subcase 2.1: $S_H=\{0\}$.}

Then, we directly obtain $\dim \calS \leq \dbinom{n}{2}$.

\subsubsection{Subcase 2.2: $S_H$ contains a non-alternating matrix.}

Then, working like in Section \ref{section5.1.3}, we obtain
$$\dim \calS \leq \dim S_H+\max\Bigl(s^{(1)}_{n-1,n-2}, s^{(2)}_{n-1,n-2}\Bigr)
\leq n-1+\dbinom{n-1}{2}=\dbinom{n}{2}.$$

\subsubsection{Subcase 2.3: $S_H$ is non-zero and contains only alternating matrices.}

Then, $\dim S_H \leq n-2$ since $\Mata_n(\K)_H \not\subset S_H$. Using the same line of reasoning as in Section \ref{section4.2},
we deduce that
\begin{align*}
\dim \calS & \leq \dim S_H+(n-1)+\max\left(s^{(1)}_{n-2,n-3}, s^{(2)}_{n-2,n-3}\right) \\
& \leq (n-2)+(n-1)+\dbinom{n-2}{2}=\dbinom{n}{2}.
\end{align*}
The inequality statement in Theorem \ref{maintheosymmetric} is now established in all situations.

\subsection{Some corollaries}\label{corollsymmetricsection}

As in Section \ref{corollalternatingsection}, we obtain the
following two corollaries of the inequality statement in Theorem \ref{maintheosymmetric}.

\begin{cor}\label{generationcor1symmetric}
Let $n$ and $r$ be non-negative integers, with $r<n$.
Let $\calS$ be an affine subspace of $\Mats_n(\K)$ such that
$$\dim \calS \geq 2+\max\bigl(s^{(1)}_{n,r},s^{(2)}_{n,r}\bigr).$$
Then, the affine space $\calS$ is generated by its matrices with rank greater than $r$.
\end{cor}

\begin{cor}\label{generationcor2symmetric}
Let $n$ be a non-negative integer. Then, the affine space $\Mats_n(\K)$ is
 generated by its non-singular matrices unless $n=1$ and $\# \K=2$.
\end{cor}

\section{Spaces of alternating matrices with the maximal dimension}\label{section6}

In this section, we prove the second statement in Theorem \ref{maintheoalternate}, that is we
determine the subspaces of $\Mata_n(\K)$ with upper-rank $r=2s<n$ and with the critical dimension
$\max\bigl(a^{(1)}_{n,r}, a^{(2)}_{n,r}\bigr)$.
Once more, the proof is done by induction over $n$ and $r$.

The case $n<3$ is trivial, and from now on we assume that $n \geq 3$.
Let $r=2s$ be an even integer such that $0 \leq r<n$.

Let $\calS$ be an affine subspace of $\Mata_n(\K)$ such that
$$\urk \calS \leq r \quad \text{and} \quad
\dim \calS=\max\left(a^{(1)}_{n,r}, a^{(2)}_{n,r}\right).$$
We wish to prove that $\calS$ is congruent to $\widetilde{\Mata_{r+1}(\K)}^{(n)}$ or $\WA_{n,r}(\K)$, or
that $\# \K=2$, $n=4$ and $\calS$ is congruent to $\calU(\K)$.
The case $r=0$ is trivial. If $r=n-1$ then we see that $a^{(2)}_{n,r} \leq a^{(1)}_{n,r}$; then,
$\dim \calS=\dbinom{n}{2}=\dim \Mata_n(\K)$ and it follows that $\calS=\Mata_n(\K)$.
In the rest of the proof, we assume that $1 \leq r \leq n-2$.

To prove the claimed statement, we come right back to the line of reasoning of Section \ref{section3}.
We lose no generality in assuming that, for the linear hyperplane $H=\K^{n-1} \times \{0\}$, we have $\dim S_H\leq s$.
From there, we split the discussion along the form of $S_H$.

\subsection{Case 1: $S_H=\{0\}$.}\label{section6.1}

We split every matrix $M$ in $\calS$ up as
$$M=\begin{bmatrix}
P(M) & C(M) \\
-C(M)^T & 0
\end{bmatrix} \quad \text{with $P(M) \in \Mata_{n-1}(\K)$ and $C(M) \in \K^{n-1}$.}$$
Then, $P(\calS)$ is an affine subspace of $\Mata_{n-1}(\K)$ with upper-rank less than or equal to $r$
and
$$\dim \calS=\dim P(\calS).$$
By the first statement of Theorem \ref{maintheoalternate}, we know that
$$\dim P(\calS) \leq \max\bigl(a_{n-1,r}^{(1)},a_{n-1,r}^{(2)}\bigr).$$
Inequality $a_{n-1,r}^{(2)} \geq a_{n-1,r}^{(1)}$ would  lead to $\dim \calS \leq a_{n-1,r}^{(2)}<a_{n,r}^{(2)}$, in contrast with our assumptions.
Thus, $a_{n-1,r}^{(1)}>a_{n-1,r}^{(2)}$.
In particular, we do not have $(n-1,r)=(4,2)$.
By induction, we deduce that $P(\calS)$ is congruent to
$\widetilde{\Mata_{r+1}(\K)}^{(n)}$.
Without further loss of generality, we can assume that $P(\calS)=\widetilde{\Mata_{r+1}(\K)}^{(n)}$.

Now, as $S_H=\{0\}$ the factorization lemma yields affine
mappings $C_1 : \Mata_{r+1}(\K) \rightarrow \K^{r+1}$ and $C_2 : \Mata_{r+1}(\K) \rightarrow \K^{n-r-2}$
such that $\calS$ is the set of all matrices of the form
$$M(A)=\begin{bmatrix}
A & [0]_{(r+1) \times (n-r-2)} & C_1(A) \\
[0]_{(n-r-2)\times (r+1)} & [0]_{(n-r-2)\times (n-r-2)} & C_2(A) \\
-C_1(A)^T & -C_2(A)^T & 0
\end{bmatrix} \quad \text{with $A \in \Mata_{r+1}(\K)$.}$$
For all $A \in \Mata_{r+1}(\K)$, we see that $\rk M(A) \geq \rk A+2$ if $C_2(A) \neq 0$, and
hence $C_2$ vanishes at every rank $r$ matrix of $\Mata_{r+1}(\K)$.
As $r\geq 2$, Corollary \ref{generationcor2alternating} shows that the set of all rank $r$ matrices of $\Mata_{r+1}(\K)$
generates the affine space $\Mata_{r+1}(\K)$, whence $C_2=0$.

\begin{claim}\label{claimRC1}
If $r>2$ or $\# \K>2$ then $C_1$ is range-compatible (and hence, linear). \\
If $r=2$ and $\# \K=2$, then $C_1$ maps every \emph{non-zero} matrix of $\Mata_3(\K)$ to a vector of its range.
\end{claim}

\begin{proof}
We have an affine map $\varphi : \Mata_r(\K) \rightarrow \K$ such that, for all $B \in \Mata_r(\K)$,
the scalar $\varphi(B)$ is the last entry of $C_1\bigl(B \oplus 0_1\bigr)$.
Yet, $\rk M\bigl(B \oplus 0_1\bigr)>r$ if $\rk B=r$ and $\varphi(B) \neq 0$.
It follows that $\varphi$ vanishes at every rank $r$ matrix of $\Mata_r(\K)$.

Assume that $r \geq 3$ or $\# \K>2$. Then, by Corollary \ref{generationcor2alternating} we deduce that
$\varphi=0$. In other words $C_1(A) \in \K^r \times \{0\}$ whenever $\im A \subset \K^r \times \{0\}$.
Using congruence transformations we can generalize this as follows: for any linear hyperplane $V$ of $\K^{r+1}$
and any matrix $A \in \Mata_{r+1}(\K)$, the inclusion $\im A \subset V$ implies $C_1(A) \in V$.
Now, if we let $A \in \Mata_{r+1}(\K)$, we can write $\im A$ as the intersection of a family of linear hyperplanes
$V_1,\dots,V_p$ of $\K^{r+1}$, and hence $C_1(A) \in \underset{i=1}{\overset{p}{\bigcap}} V_i=\im A$.
Hence, $C_1$ is range-compatible. In particular, $C_1(0)=0$, and as $C_1$ is affine it is actually linear.

Assume finally that $r=2$ and $\# \K=2$. Then, with the same congruence argument as before, we obtain that
$C_1$ maps every rank $2$ matrix of $\Mata_3(\K)$ to a vector of its range.
\end{proof}

Assume that $r>2$ or $\# \K>2$, or that $C_1$ is linear. Then, $C_1$ is range-compatible, and
Theorem \ref{alternatingRC} yields a vector $X \in \K^{r+1}$ such that $C_1(A)=AX$ for all $A \in \Mata_{r+1}(\K)$.
Setting finally $Q:=\begin{bmatrix}
I_{r+1} & [0]_{(r+1) \times (n-r-2)} & -X \\
[0]_{(n-r-2) \times (r+1)} & I_{n-r-2} & [0]_{(n-r-2)\times 1}  \\
[0]_{1 \times (r+1)} & [0]_{1 \times (n-r-2)} & 1
\end{bmatrix}$, we
see that
$$Q^T\, \calS\, Q=\widetilde{\Mata_{r+1}(\K)}^{(n)},$$
which is the desired congruence.

Assume now that $r=2$, $\# \K=2$ and $C_1$ is non-linear.
As $r=2$ we must have $n=4$ (indeed, $a^{(2)}_{n,2}>3$ whenever $n>4$).
Moreover, $C_1(0) \neq 0$. Performing an additional harmless congruence transformation, we can assume that
$C_1(0)=\begin{bmatrix}
1 \\
0 \\
0
\end{bmatrix}$.
Then, we have a \emph{linear} map $\psi : \K^2 \rightarrow \K^2$ and an affine form $\chi : \K^2 \rightarrow \K$ such that
$$\forall X \in \K^2, \; C_1\left(\begin{bmatrix}
0 & -X^T \\
X & [0]_{2 \times 2}
\end{bmatrix}\right) =\begin{bmatrix}
\chi(X) \\
\psi(X)
\end{bmatrix}.$$
The vector $\psi(X)$ is collinear to $X$ for all $X \in \K^2$; this yields a scalar $\lambda$ such that
$\psi : X \mapsto \lambda X$.
Thanks to another harmless congruence transformation, we can assume that $\lambda=0$ and that $\chi$ is constant.
Then, we find scalars $\alpha,\beta,\gamma$ such that
$$C_1 : \begin{bmatrix}
0 & a & b \\
a & 0 & c \\
b & c & 0
\end{bmatrix} \longmapsto \begin{bmatrix}
\alpha c+1 \\
\beta c \\
\gamma c
\end{bmatrix}.$$
With $a=b=0$ and $c=1$, we deduce that $\alpha=1$. \\
With $a=c=1$ and $b=0$, we obtain $\gamma=0$. With $b=c=1$ and $a=0$, we obtain $\beta=0$.
Then, we conclude that $\calS=\calU(\K)$.

\subsection{Case 2: $S_H \neq \{0\}$.}\label{section6.2}

As in Section \ref{case2dimensionalternating}, we lose no generality in assuming that $S_H$ contains $E_{1,n}-E_{n,1}$,
and then we split every matrix $M$ of $\calS$ up as
$$M=\begin{bmatrix}
0 & [?]_{1 \times (n-2)} & ? \\
[?]_{(n-2) \times 1} & K(M) & [?]_{(n-2) \times 1} \\
? & [?]_{1 \times (n-2)} & 0
\end{bmatrix}.$$
Then, we find that
$$\urk K(\calS) \leq r-2 \quad \text{and} \quad \dim \calS \leq \dim K(\calS)+(n-2)+s.$$
By the inequality statement in Theorem \ref{maintheoalternate}, we have
$$\dim K(\calS) \leq \max\left(a^{(1)}_{n-2,r-2}, a^{(2)}_{n-2,r-2}\right).$$

Assume that $a^{(1)}_{n-2,r-2} \geq a^{(2)}_{n-2,r-2}$.
Then, either $s-1=0$ or $5(s-1) \geq 2(n-2)-3$, and in the latter case the line of reasoning from Section \ref{case2dimensionalternating}
would yield $\dim \calS < a^{(1)}_{n,r}$, contradicting our assumptions.
Thus, $s=1$, and hence $K(\calS)=\{0\}=\WA_{n-2,r-2}(\K)$.

Assume now that $a^{(1)}_{n-2,r-2}<a^{(2)}_{n-2,r-2}$, so that $(n-2,r-2) \neq (4,2)$.
Then, by the rank theorem,
$$\dim \calS \leq \dim K(\calS)+(n-2)+\dim S_H \leq a^{(2)}_{n-2,r-2}+(n-2)+s=a^{(2)}_{n,r} \leq \dim \calS,$$
which shows that $\dim K(\calS)=a^{(2)}_{n-2,r-2}$ and $\dim \calS=a^{(2)}_{n,r}$. Then,
as $a^{(1)}_{n-2,r-2}<a^{(2)}_{n-2,r-2}$ and $(n-2,r-2) \neq (4,2)$,
we find by induction that $K(\calS)$ is congruent to $\WA_{n-2,r-2}(\K)$.

Therefore, in any case $K(\calS)$ is congruent to $\WA_{n-2,r-2}(\K)$.
We lose no generality in assuming that $K(\calS)=\WA_{n-2,r-2}(\K)$, and from now on we assume that this holds.

Now, we can split every matrix of $\calS$ up as
$$M=\begin{bmatrix}
A(M) & -B(M)^T & [?]_{s \times 1} \\
B(M) & [0]_{(n-s-1) \times (n-s-1)} & C(M) \\
[?]_{1 \times s} & -C(M)^T & 0
\end{bmatrix}$$
where $A(M) \in \Mata_s(\K)$, $B(M) \in \Mat_{n-s-1,s}(\K)$ and $C(M) \in \K^s$.

Set
$$\calV:=\biggl\{\begin{bmatrix}
A(M) & -B(M)^T  \\
B(M) & [0]_{(n-s-1) \times (n-s-1)}
\end{bmatrix} \mid M \in \calS \biggr\}$$
and
$$\calT:=\Bigl\{\begin{bmatrix}
B(M) & C(M)
\end{bmatrix} \mid M \in \calS \Bigr\} \subset \Mat_{n-s-1,s+1}(\K).$$
Note that $\calV$ is a subspace of $\WA_{n-1,r}(\K)$, and
$$\dim \calV = \dim \calS-\dim S_H \geq a_{n,r}^{(2)}-s=a_{n-1,r}^{(2)}.$$
In turn, this shows that
$$\calV=\WA_{n-1,r}(\K),$$
and we deduce that $B(\calS)=\Mat_{n-s-1,s}(\K)$.
For every $M \in \calS$, we see by a standard rank computation that
$$\rk M \geq \rk \begin{bmatrix}
B(M) & C(M)
\end{bmatrix}+\rk \begin{bmatrix}
-B(M)^T \\
-C(M)^T
\end{bmatrix}=2 \rk \begin{bmatrix}
B(M) & C(M)
\end{bmatrix}.$$
It follows that $\urk \calT \leq s$. Moreover $s+1 \leq n-s-1$.
Then, we can try to apply Corollary \ref{Flanderscor} to $\calT$.
Note that if the assumptions of this result are not satisfied by $\calT$, we find
$n-s-1=s+1=2$, $\# \K=2$ and $\calT$ does not contain the zero matrix, the former of which leads to
$s=1$ and $n=4$.

\vskip 3mm
\noindent \textbf{Case a: Corollary \ref{Flanderscor} applies to $\calT$.} \\
Then, we obtain a vector $Y \in \K^s$ such that $C(M)=B(M)Y$ for all $M \in \calS$.
Setting
$$Q:=\begin{bmatrix}
I_s & [0]_{s \times (n-1-s)} & -Y \\
[0]_{(n-1-s) \times s} & I_{n-1-s} & [0]_{(n-1-s) \times 1} \\
[0]_{1 \times s} & [0]_{1 \times (n-1-s)} & 1
\end{bmatrix},$$
we see that replacing $\calS$ with $Q^T \calS Q$ affects none of the previous
assumptions but in that reduced situation we also have $C(M)=0$ for all $M \in \calS$.
Thus, in that situation $\calS \subset \WA_{n,r}(\K)$. As
$\dim \calS \geq a_{n,r}^{(2)}=\dim \WA_{n,r}(\K)$, we conclude that $\calS=\WA_{n,r}(\K)$.

\vskip 3mm
\noindent \textbf{Case b: $n=4$, $s=1$, $\# \K=2$ and $\calT$ is not a linear subspace of $\Mat_2(\K)$.} \\
As $M \in \calS \mapsto B(M) \in \K^2$ is surjective we learn that there is a matrix of
the form $\begin{bmatrix}
[0]_{2 \times 1} & X
\end{bmatrix}$ in $\calT$. Then, $X \neq 0$ since $\calT$ is not a linear subspace of $\Mat_2(\K)$.
Using an additional congruence transformation we can reduce the situation to the one where $X=\begin{bmatrix}
1 \\
0
\end{bmatrix}$. Yet $B(\calS)=\K^2$, whereas Flanders's theorem shows that $\dim \calT \leq 2$. Hence,
there are scalars $\alpha,\beta,\gamma,\delta$ such that
$$\calT=\left\{\begin{bmatrix}
x & \alpha x+\beta y+1 \\
y & \gamma x+\delta y
\end{bmatrix}\mid (x,y)\in \K^2\right\}.$$
Computing the determinant yields
$$\forall (x,y)\in \K^2, \; (\delta+\alpha)xy+\gamma x+(\beta+1) y=0,$$
which leads to $\gamma=0$, $\delta=\alpha$ and $\beta=1$. Performing an additional congruence transformation then
leaves us with the case when
$$\calT=\left\{\begin{bmatrix}
a & b+1 \\
b & 0
\end{bmatrix} \mid (a,b)\in \K^2\right\}.$$
Hence,
$$\calS \subset \left\{\begin{bmatrix}
0 & a & b & c \\
a & 0 & 0 & b+1 \\
b & 0 & 0 & 0 \\
c & b+1 & 0 & 0
\end{bmatrix}\mid (a,b,c)\in \K^3\right\}.$$
As the dimension of each space equals $3$, we deduce that they are equal.
Finally, using the congruence transformation $M \mapsto QMQ^T$ for
$Q:=\begin{bmatrix}
1 & 0 & 0 & 0 \\
0 & 1 & 0 & 0 \\
0 & 0 & 0 & 1 \\
0 & 0 & 1 & 0
\end{bmatrix}$, we conclude that $\calS$ is congruent to $\calU(\K)$.

This completes our proof of Theorem \ref{maintheoalternate}.

\section{Spaces of symmetric matrices with the maximal dimension: the characteristic not $2$ case}\label{section7}

In this section, we prove the second statement in Theorem \ref{maintheosymmetric} for fields with characteristic not $2$.
In other words, by induction over $n$ and $r$, we shall classify the affine subspaces of $\Mats_n(\K)$ with upper-rank $r<n$ and with the critical dimension
$\max\left(s^{(1)}_{n,r}, s^{(2)}_{n,r}\right)$.

Let $n,r$ be non-negative integers such that $r<n$, and assume that $\K$ has characteristic not $2$.
Throughout the proof, we set
$$s:=\left \lfloor \frac{r}{2}\right\rfloor.$$
Let $\calS$ be an affine subspace of $\Mats_n(\K)$ such that
$$\urk \calS \leq r \quad \text{and} \quad \dim \calS=\max\bigl(s^{(1)}_{n,r}, s^{(2)}_{n,r}\bigr).$$
We wish to prove that $\calS$ is congruent to $\widetilde{\Mats_r(\K)}^{(n)}$ or to $\WS_{n,r}(\K)$.
The case $r=0$ is trivial, and the case $r=1$ has been dealt with in Proposition \ref{rankatmost1lemma}.
Thus, in the rest of the proof we assume that $r \geq 2$.

To prove the claimed statement, we come right back to the line of reasoning of Section \ref{section4}.
Denote by $m$ the minimal dimension among the spaces of type $S_H$, where $H$ ranges over the linear hyperplanes of $\K^n$.
By Lemma \ref{keylemmasymmetriccarnot2}, we have
$$m \leq s.$$
Without loss of generality, we can now assume that $S_H=m$ for $H:=\K^{n-1} \times \{0\}$.
From there, we split the discussion into three subcases, whether $S_H=\{0\}$ or $S_H$ contains a rank $2$ matrix
or $S_H$ is non-zero and contains no rank $2$ matrix.

\subsection{Case 1: $S_H=\{0\}$.}\label{section7.1}

We can split every matrix of $\calS$ up as
$$M=\begin{bmatrix}
P(M) & [?]_{(n-1) \times 1} \\
[?]_{1 \times (n-1)} & ?
\end{bmatrix}\quad \text{with $P(M) \in \Mats_{n-1}(\K)$.}$$
Then $P(\calS)$ is an affine subspace of $\Mat_{n-1}(\K)$ with
$$\urk P(\calS) \leq r \quad \text{and} \quad
\dim P(\calS)=\dim \calS=\max\left(s_{n,r}^{(1)}, s_{n,r}^{(2)}\right)\geq \max\left(s_{n-1,r}^{(1)}, s_{n-1,r}^{(2)}\right).$$
Hence, by the inequality statement in Theorem \ref{maintheosymmetric}, we have
$$\dim P(\calS)=\max\left(s_{n-1,r}^{(1)}, s_{n-1,r}^{(2)}\right).$$
Assume first that $r<n-1$. If $\dim P(\calS)=s_{n-1,r}^{(2)}$ then $\dim \calS<s_{n,r}^{(2)}$ since $r \geq 2$, contradicting our assumptions.
Hence, $\dim P(\calS)=s_{n-1,r}^{(1)}>s_{n-1,r}^{(2)}$, and by induction we find that
$P(\calS)$ is congruent to $\widetilde{\Mats_r(\K)}^{(n-1)}$.

If $r=n-1$, then, as $\dim P(\calS) \leq \dim \Mats_{n-1}(\K)=s_{n-1,r}^{(1)}$, we readily have
$P(\calS)=\Mats_{n-1}(\K)=\widetilde{\Mats_r(\K)}^{(n-1)}$.

Thus, in any case no generality is lost in assuming that
$$P(\calS)=\widetilde{\Mats_r(\K)}^{(n-1)}.$$
In that situation, the fact that $S_H=\{0\}$ yields affine mappings
$$C_1 : \Mats_r(\K) \rightarrow \K^r, \; C_2 : \Mats_r(\K) \rightarrow \K^{n-r-1}, \; b : \Mats_r(\K) \rightarrow \K$$
such that $\calS$ is the set of all matrices of the form
$$\begin{bmatrix}
N & [0]_{r \times (n-r-1)} & C_1(N) \\
[0]_{(n-r-1) \times r} & [0]_{(n-r-1) \times (n-r-1)} & C_2(N) \\
C_1(N)^T & C_2(N)^T & b(N)
\end{bmatrix} \quad \text{with $N \in \Mats_r(\K)$.}$$
Let $N \in \Mats_r(\K)$ be with rank $r$: then, as $\urk \calS \leq r$ we obtain
$C_2(N)=0$. As $r \geq 1$ and $\K$ has more than $2$ elements, Corollary \ref{generationcor2symmetric}
yields $C_2=0$.

From there, one uses the same line of reasoning as in Section \ref{section6.1}.
As the affine space $\Mats_r(\K)$ is generated  by its rank $r$ matrices (again, by Corollary \ref{generationcor2symmetric}),
one uses the same chain of arguments as in Claim \ref{claimRC1} to obtain:

\begin{claim}\label{claimRC2}
The map $C_1$ is range-compatible (and hence, linear).
\end{claim}

Then, we deduce from point (a) of Theorem \ref{symmetricRC} that $C_1 : N \mapsto NX$ for some $X \in \K^r$.
By setting
$$Q:=\begin{bmatrix}
I_r & [0]_{r \times (n-r-1)} & -X \\
[0]_{(n-r-1) \times r} & I_{n-r-1} & [0]_{(n-r-1)\times 1} \\
[0]_{1 \times r} & [0]_{1 \times (n-r-1)} & 1
\end{bmatrix},$$
it follows that the space $Q^T \calS Q$ has the same form as $\calS$, with the new $C_1$ map equal to zero.

Thus, no generality is lost in assuming that $C_1=0$. In that situation, we note that $b(N)=0$ for every invertible matrix $N \in \Mats_r(\K)$, which,
by using Corollary \ref{generationcor2symmetric} once more, proves that $b=0$. Hence, in that reduced situation, $\calS=\widetilde{\Mats_r(\K)}^{(n)}$.

\subsection{Case 2: $S_H$ contains a rank $2$ matrix.}\label{section7.2}

As in Section \ref{section4.2}, we can assume that the space $S$ contains a matrix of the form
$$N=\begin{bmatrix}
0 & [0]_{1 \times (n-2)} & 1 \\
[0]_{(n-2) \times 1} & [0]_{(n-2) \times (n-2)} & [0]_{(n-2) \times 1} \\
1 & [0]_{1 \times (n-2)} & ?
\end{bmatrix}.$$
Then, we split every matrix $M \in \calS$ up as
$$M=\begin{bmatrix}
? & [?]_{1 \times (n-2)} & ? \\
[?]_{(n-2) \times 1} & K(M) & [?]_{(n-2) \times 1} \\
? & [?]_{1 \times (n-2)} & ?
\end{bmatrix} \quad \text{with $K(M) \in \Mats_{n-2}(\K)$,}$$
and we obtain that $K(\calS)$ is an affine subspace of $\Mats_{n-2}(\K)$ with $\urk K(\calS) \leq r-2$.
By the inequality statement in Theorem \ref{maintheosymmetric}, we have
$$\dim K(\calS) \leq \max\left(s_{n-2,r-2}^{(1)},s_{n-2,r-2}^{(2)}\right).$$

\begin{claim}
The space $K(\calS)$ is congruent to $\WS_{n-2,r-2}(\K)$, and
$\dim \calS=s^{(2)}_{n,r.}$
\end{claim}

\begin{proof}
Assume first that $r \not\in \{2,3\}$.
If $s^{(1)}_{n-2,r-2} \geq s^{(2)}_{n-2,r-2}$,
then $5(s-1) \geq 2(n-2)-1$ if $r$ is even, otherwise $5(s-1) \geq 2(n-2)-5$; in any case,
by following the line of reasoning from Section \ref{section4.2} we find that
$$\dim \calS <s^{(1)}_{n,r},$$
contradicting our assumptions.
It follows that $s^{(1)}_{n-2,r-2} < s^{(2)}_{n-2,r-2.}$
Thus, $\dim K(\calS) \leq s^{(2)}_{n-2,r-2}$, and by using the
chain of inequalities
$$\dim \calS \leq \dim K(\calS)+(n-1)+\dim S_H \leq s^{(2)}_{n-2,r-2}+(n-1)+s=s^{(2)}_{n,r} \leq \dim \calS,$$
we obtain that
$$\dim \calS=s^{(2)}_{n,r} \quad \text{and} \quad \dim K(\calS)=s^{(2)}_{n-2,r-2.}$$
Then, as $s^{(1)}_{n-2,r-2} < s^{(2)}_{n-2,r-2}$ we find by induction that $K(\calS)$ is congruent to
$\WS_{n-2,r-2}(\K)$.

Finally, if $r\in \{2,3\}$ then we see that $s^{(1)}_{n-2,r-2}=s^{(2)}_{n-2,r-2}$, and with the same line of reasoning
as above we find that $\dim \calS=s^{(2)}_{n,r}$ and $\dim K(\calS)=s^{(2)}_{n-2,r-2}$.
Since $\widetilde{S_1(\K)}^{(n)}=\WS_{n,1}(\K)$ and $\widetilde{S_0(\K)}^{(n)}=\WS_{n,0}(\K)$,
the induction hypothesis still yields that $K(\calS)$ is congruent to $\WS_{n-2,r-2}(\K)$.
\end{proof}

Thus, no generality is lost in assuming that
$$K(\calS)=\WS_{n-2,r-2}(\K).$$
Now, writing every matrix $M$ of $\calS$ as
$$M=\begin{bmatrix}
P(M) & [?]_{(n-1) \times 1} \\
[?]_{1 \times (n-1)} & ?
\end{bmatrix},$$
we see that $P(\calS) \subset \WS_{n-1,r}(\K)$, and on the other hand we have
$$s_{n,r}^{(2)} \leq \dim \calS=\dim P(\calS)+\dim S_H \leq s_{n-1,r}^{(2)}+s=s_{n,r}^{(2)}.$$
Hence, $\dim P(\calS)=s_{n-1,r}^{(2)}$, and it follows that
$$P(\calS)=\WS_{n-1,r}(\K) \quad \text{and} \quad m=\dim S_H=s.$$
From there, we split the discussion into two subcases, whether $r$ is even or odd.

\subsubsection{Subcase 2.1: $r$ is even.}\label{section7.2.1}

Then, we split every matrix $M$ of $\calS$ up as
$$M=\begin{bmatrix}
[?]_{s \times s} & B(M)^T & [?]_{s \times 1} \\
B(M) & [0]_{(n-1-s) \times (n-1-s)} & C(M) \\
[?]_{1 \times s} & C(M)^T & a(M)
\end{bmatrix}$$
with $a(M) \in \K$, $B(M) \in \Mat_{n-1-s,s}(\K)$ and $C(M) \in \K^s$.
Since $P(\calS)=\WS_{n-1,r}(\K)$, we find that $B(\calS)=\Mat_{n-1-s,s}(\K)$.
As in Section \ref{section6.2}, we obtain that the affine space
$$\calT:=\Bigl\{\begin{bmatrix}
B(M) & C(M)
\end{bmatrix}\mid M \in \calS\Bigr\}\subset \Mat_{n-1-s,s+1}(\K)$$
satisfies $\urk \calT \leq s$.
Yet, as $\dim \calS=s_{n,r}^{(2)}$ we have $s_{n,r}^{(2)} \geq s_{n,r}^{(1)}$ by the inequality statement from Theorem \ref{maintheosymmetric},
and hence $n \geq \frac{5s+1}{2}\cdot$
If $n-1-s<s+1$ then $\frac{5s+1}{2}-(2s+1)\leq 0$, which, as $s>0$, leads to $s=1$ and $n=3$.

Assume first that $(n,s) \neq (3,1)$.
Then, $n-1-s \geq s+1$ and, as $\# \K>2$ (remember that the characteristic of $\K$ does not equal $2$),
we can follow the line of reasoning from Section \ref{section6.2} to see that, after applying a carefully chosen congruence transformation to $\calS$,
the situation can be reduced to the one where $C=0$.

Then, for all $i \in \lcro s+1,n-1\rcro$, we define $H_i$ as the linear hyperplane associated with the equation $x_i=0$ in the standard basis,
and we see that $S_{H_i} \subset \Vect(E_{i,j}+E_{j,i})_{1 \leq j \leq s.}$ As $m=s$
we deduce that $S_{H_i} = \Vect(E_{i,j}+E_{j,i})_{1 \leq j \leq s.}$

Let $M \in \calS$.
For all $(t_1,\dots,t_s) \in \K^s$, the space $\calS$ contains
$M+\underset{k=1}{\overset{s}{\sum}} t_k (E_{s+k,k}+E_{k,s+k})$ (because $n-s-1 \geq s$) and hence this matrix has rank less than $2s+1$.
Applying Lemma \ref{symmetriccornerlemma1} inductively, we deduce that
$$\rk \begin{bmatrix}
[0]_{(n-2s-1) \times (n-2s-1)} & [0]_{(n-2s-1) \times 1} \\
[0]_{1 \times (n-2s-1)} & a(M)
\end{bmatrix} \leq 2s-2s=0,$$
whence $a(M)=0$.
Thus, $\calS \subset \WS_{n,r}(\K)$. As $s_{n,r}^{(2)}=\dim \calS=\dim \WS_{n,r}(\K)$,
we conclude that $\calS=\WS_{n,r}(\K)$.

Assume finally that $n=3$ and $s=1$.
Remember that $S_H$ contains $E_{1,3}+E_{3,1}+x E_{3,3}$ for some $x \in \K$.
Using a harmless congruence transformation, we can actually assume that $S_H$ contains $E_{1,3}+E_{3,1}$.
As $\dim S_H \leq 1$, we deduce that $S_H=\Vect(E_{1,3}+E_{3,1})$. As $P(\calS)=\WS_{2,1}(\K)$, this yields scalars
$\alpha,\beta,\gamma,\lambda,\mu,\nu$ such that
$$\calS=\Biggl\{\begin{bmatrix}
a & b & c \\
b & 0 & \alpha a+\beta b+\gamma \\
c & \alpha a+\beta b+\gamma & \lambda a+\mu b+\nu
\end{bmatrix}\mid (a,b,c)\in \K^3\Biggr\}.$$
If $\mu$ were non-zero, then, for the hyperplane $H_2:=\K \times \{0\} \times \K$ of $\K^3$,
we would have $S_{H_2}=\{0\}$, contradicting $m=s=1$.

Hence, $\mu=0$. Taking $b=1$ and $a=c=0$ then leads to $\nu=0$.
Taking $b=c=0$ and $a=t$ for $t \in \K\setminus \{0\}$, we find
$\alpha t+\gamma=0$, which leads to $\alpha=\gamma=0$ since $\# \K>2$.
Finally, we compute the determinant in the general case to obtain
$$\forall (a,b,c)\in \K^3, \; 2\beta b^2c-(\lambda+\beta^2)ab^2=0.$$
On the left hand-side we have a polynomial of degree at most $2$ in each variable, and as $\#\K>2$ and $\K$ has characteristic not $2$
we deduce that $\beta=0$ and $\lambda=-\beta^2=0$. Therefore, $\calS=\WS_{3,2}(\K)$.

\subsubsection{Subcase 2.2: $r$ is odd.}\label{section7.2.2}

Let us split every matrix $M$ of $\calS$ up as
$$M=\begin{bmatrix}
[?]_{s \times s} & [?]_{s \times 1} & B(M)^T & [?]_{s \times 1} \\
[?]_{1 \times s} & a(M) & [0]_{1 \times (n-s-2)} & b(M) \\
B(M) & [0]_{(n-s-2) \times 1} & [0]_{(n-s-2) \times (n-s-2)} & C(M) \\
[?]_{1 \times s}  & b(M) & C(M)^T & c(M)
\end{bmatrix}$$
with $B(M) \in \Mat_{n-2-s,s}(\K)$, $C(M) \in \K^{n-2-s}$ and scalars $a(M)$, $b(M)$ and $c(M)$.
As $P(\calS)=\WS_{n-1,r}(\K)$ we note that $B(\calS)=\Mat_{n-s-2,s}(\K)$.

Let us consider the affine space
$$\calT:=\Bigl\{  \begin{bmatrix}
B(M) & C(M)
\end{bmatrix} \mid M \in \calS\Bigr\} \subset \Mat_{n-s-2,s+1}(\K).$$
For all $M$ in $\calS$, we see by standard rank computations that
$$\rk M \geq \rk \begin{bmatrix}
B(M) & C(M)
\end{bmatrix}+\rk \begin{bmatrix}
B(M)^T \\
C(M)^T
\end{bmatrix}=2 \rk \begin{bmatrix}
B(M) & C(M)
\end{bmatrix},$$
and since $\rk M \leq 2s+1$ this leads to $\rk \begin{bmatrix}
B(M) & C(M)
\end{bmatrix} < s+1$.
Therefore,
$$\urk \calT \leq s.$$
Again, as $\dim \calS=s_{n,r}^{(2)}$, the inequality statement in Theorem \ref{maintheosymmetric} leads to
$s_{n,r}^{(2)} \geq s_{n,r}^{(1)}$. As $s>0$, this shows that
$5s \leq 2n-5$, successively leading to $n-2s-3 \geq \frac{s-1}{2} \geq 0$ and to $n-s-2 \geq s+1 \geq 2$.

Since $B(\calS)=\Mat_{n-s-2,s}(\K)$, we can apply Corollary \ref{Flanderscor}, just like in Section \ref{section6.2},
to reduce the situation to the one where $C=0$ (note here that $\# \K>2$ since $\K$ has characteristic not $2$).

Let $i \in \lcro s+2,n-1\rcro$, and consider the hyperplane $H_i$ of $\K^n$ defined by the equation $x_i=0$ in the canonical basis.
As in the previous case, we use the minimality of $m$ to obtain that
$S_{H_i}=\Vect(E_{i,j}+E_{j,i})_{1 \leq j \leq s.}$

Let $M \in \calS$ and set
$$J(M):=\begin{bmatrix}
a(M) & b(M) \\
b(M) & c(M)
\end{bmatrix} \in \Mat_2(\K).$$
For all $(t_1,\dots,t_s) \in \K^s$, the space $\calS$ contains
$M+\underset{k=1}{\overset{s}{\sum}} t_k (E_{s+k+1,k}+E_{k,s+k+1})$ (because $n-s-2 \geq s$) and hence this matrix has rank at most $2s+1$.
Applying Lemma \ref{symmetriccornerlemma1} inductively, we deduce that
$$\rk \begin{bmatrix}
a(M) & [0]_{1 \times (n-2s-2)} & b(M) \\
[0]_{(n-2s-2)\times 1} & [0]_{(n-2s-2) \times (n-2s-2)} & [0]_{(n-2s-2) \times 1} \\
b(M) & [0]_{1 \times (n-2s-2)} & c(M)
\end{bmatrix} \leq 2s+1-2s=1,$$
whence
$$\rk J(M) \leq 1.$$

It follows that $J(\calS)$ is an affine subspace of $\Mat_2(\K)$ with upper-rank at most $1$.
On the other hand $\dim J(\calS) \geq 1$ since $P(\calS)=\WS_{n-1,r}(\K)$.
By the inequality statement in Theorem \ref{maintheosymmetric}, we obtain $\dim J(\calS)=1$, and it follows from
Proposition \ref{rankatmost1lemma} that $J(\calS)$ is congruent to $\Vect(E_{1,1})$ (remember that $\# \K>2$ since $\K$ has characteristic not $2$).

It follows that $\calS$ is congruent to a subspace of $\WS_{n,r}(\K)$.
Since $\dim \calS=s_{n,r}^{(2)}=\dim \WS_{n,r}(\K)$, we conclude that $\calS$ is congruent to $\WS_{n,r}(\K)$.

\subsection{Case 3: The space $S_H$ is non-zero and contains no rank $2$ matrix.}\label{section7.3}

Thus, $S_H=\Vect(E_{n,n})$ and $m=1$.
As in Section \ref{section4.2}, it follows from the inequality statement in Theorem \ref{maintheosymmetric} that
$$\dim \calS-1 \leq s_{n-1,r-1}^{(1)} \quad \text{or} \quad \dim \calS-1 \leq s_{n-1,r-1.}^{(2)}$$
Hence,
$$s_{n,r}^{(1)}-1 \leq s_{n-1,r-1}^{(1)} \quad \text{or} \quad s_{n,r}^{(2)}-1 \leq s_{n-1,r-1.}^{(2)}$$
Yet,
$$s_{n,r}^{(1)}-s_{n-1,r-1}^{(1)}=r>1.$$
Moreover,
$$s_{n,r}^{(2)}-s_{n-1,r-1}^{(2)}=\begin{cases}
1+\frac{r-1}{2} & \text{if $r$ is odd} \\
n-1 & \text{if $r$ is even,}
\end{cases}$$
which, as $r \geq 2$, leads to
$$s_{n,r}^{(2)}-s_{n-1,r-1}^{(2)}>1.$$
In any case, we have found a contradiction.

This completes our inductive proof of Theorem \ref{maintheosymmetric} for fields with characteristic not $2$.

\section{Spaces of symmetric matrices with the maximal dimension: the characteristic $2$ case}\label{section8}

In this last section, we complete the proof of Theorem \ref{maintheosymmetric} by
tackling fields of characteristic $2$.
As we shall see, there is a great deal of additional complexity, in particular for the fields with two elements.
Throughout the section, we let $\K$ be an arbitrary field of characteristic $2$.
Once more, the proof is done by induction over $n$ and $r$.

Let $n$ and $r$ be non-negative integers such that $r<n$.
Let $\calS$ be an affine subspace of $\Mats_n(\K)$ such that
$$\urk \calS \leq r \quad \text{and} \quad \dim \calS=\max\left(s_{n,r}^{(1)}, s_{n,r}^{(2)}\right).$$
We wish to prove that $\calS$ is congruent to one of the spaces listed in Theorem \ref{maintheosymmetric}.
This is obvious if $r=0$, and if $r=1$ then $\dim \calS \leq 1$, whence Proposition \ref{rankatmost1lemma} yields that
$\calS$ is congruent to $\widetilde{\Mats_1(\K)}^{(n)}$ or that $\# \K=2$ and
$\calS$ is congruent to $\widetilde{Z_2(\K)}^{(n)}$.
In the remainder of the proof, we assume that $r \geq 2$, and we set
$$s:=\left\lfloor \frac{r}{2}\right\rfloor.$$
Moreover, we can assume, in the case when $n$ is odd and $r=n-1$, that $\calS \neq \Mata_n(\K)$ (for the contrary would
yield outcome (iii) in Theorem \ref{maintheosymmetric}).

Let us define $m$ as the minimal dimension for $S_H$, where $H$ ranges over
the \textbf{$\calS$-adapted} hyperplanes of $\K^n$, that is the linear hyperplanes that satisfy
$(\Mata_n(\K))_H \not\subset S_H$.
Combining Lemma \ref{keylemmasymmetriccar2} with Lemma \ref{lastkeylemma}, we obtain
that there is at least one $\calS$-adapted hyperplane of $\K^n$, and hence $m$ is well-defined
(note that $s \leq n-2$).
Moreover, $m \leq s$ provided that $r<n-1$.

Replacing $\calS$ with a congruent space, we lose no generality in assuming that $H:=\K^{n-1} \times \{0\}$
is $\calS$-adapted.
Then, we have five cases to consider:

\begin{itemize}
\item Case 1: $m=0$.
\item Case 2: $0<m \leq s$ and $S_H$ contains a non-zero alternating matrix.
\item Case 3: $m=1$ and $S_H$ contains a non-alternating matrix.
\item Case 4: $r=n-1$, $s< m$ and all the matrices of $S_H$ are alternating.
\item Case 5: $r=n-1$, $s<m$ and $S_H$ contains a non-alternating matrix.
\end{itemize}

\subsection{Case 1: $m=0$.}

Let us split every matrix of $\calS$ up as
$$M=\begin{bmatrix}
P(M) & [?]_{(n-1) \times 1} \\
[?]_{1 \times (n-1)} & ?
\end{bmatrix}\quad \text{where $P(M) \in \Mats_{n-1}(\K)$.}$$
Then $P(\calS)$ is an affine subspace of $\Mat_{n-1}(\K)$ with
$$\urk P(\calS) \leq r \quad \text{and} \quad \dim \calS=\dim P(\calS).$$
Following the line of reasoning from Section \ref{section7.1}, one shows that $\dim P(\calS)=s_{n-1,r}^{(1)}$ and
$s_{n-1,r}^{(1)}>s_{n-1,r.}^{(2)}$ In particular, if $r=2$ then we must have $n=3$.
If $r=n-1$ then we readily have $P(\calS)=\Mats_{n-1}(\K)$.
Otherwise the induction hypothesis can be applied to $P(\calS)$.
Since $(n-1,r)\neq (3,2)$ this leaves us with three possibilities for $P(\calS)$, which we regroup into two main subcases:
\begin{itemize}
\item Subcase 1.1: $r$ is even and $P(\calS)=\widetilde{\Mata_{r+1}(\K)}^{(n-1)}$;
\item Subcase 1.2: $P(\calS)=\widetilde{\Mats_r(\K)}^{(n-1)}$, or $\# \K=2$ and $P(\calS)=\widetilde{Z_{r+1}(\K)}^{(n-1)}$.
\end{itemize}
In the remainder of the section, we tackle each case separately.

\subsubsection{Subcase 1.1: $r$ is even and $P(\calS)=\widetilde{\Mata_{r+1}(\K)}^{(n-1)}$.}

We wish to prove that $\calS$ is congruent to $\widetilde{\Mata_{r+1}(\K)}^{(n)}$.
If $r=2$ then we have previously seen that $n=3$, which contradicts $r+1 \leq n-1$.
Thus, $r \geq 4$.

As $S_H=\{0\}$, there are affine maps $C_1 : \Mata_{r+1}(\K) \rightarrow \K^{r+1}$,
$C_2 : \Mata_{r+1}(\K) \rightarrow \K^{n-r-2}$, and $b : \Mata_{r+1}(\K) \rightarrow \K$
such that $\calS$ is the set of all matrices
$$\begin{bmatrix}
A & [0]_{(r+1) \times (n-r-2)} & C_1(A) \\
[0]_{(n-r-2) \times (r+1)} & [0]_{(n-r-2) \times (n-r-2)} & C_2(A) \\
C_1(A)^T & C_2(A)^T & b(A)
\end{bmatrix} \quad \text{with $A \in \Mata_{r+1}(\K)$.}$$
Then, we proceed exactly as in Section \ref{section6.1} to prove that $C_2=0$ and
that $C_1$ is range-compatible (and hence, linear): here, there is no exceptional case related to fields with cardinality $2$ because $r \geq 4$.
Then, Theorem \ref{alternatingRC} yields that $C_1$ is local, and applying an additional congruence transformation
we find that no generality is lost in assuming that $C_1=0$. Finally,
the affine form $b$ vanishes at every rank $r$ matrix of $\Mata_{r+1}(\K)$. As these matrices
generate the affine space $\Mata_{r+1}(\K)$, we deduce that $b=0$.
Hence, in this reduced situation, $\calS=\widetilde{\Mata_{r+1}(\K)}^{(n)}$.

\subsubsection{Subcase 1.2: $P(\calS)=\widetilde{\Mats_r(\K)}^{(n)}$, or $\#\K=2$ and
$P(\calS)=\widetilde{Z_{r+1}(\K)}^{(n-1)}$.}

We wish to prove that either $\calS$ is congruent to $\widetilde{\Mats_r(\K)}^{(n)}$,
or $\# \K=2$ and $\calS$ is congruent to $\widetilde{Z_{r+1}(\K)}^{(n)}$, or
$n=3$, $r=2$, $\# \K=2$ and $\calS$ is congruent to
$\calY_1(\K)$ or $\calY_2(\K)$.

Firstly, we can find affine maps $C : \Mats_r(\K) \rightarrow \Mat_{r,n-r}(\K)$ and
$D : \Mats_r(\K) \rightarrow \Mats_{n-r}(\K)$ such that
$$\calS=\biggl\{\begin{bmatrix}
N & C(N) \\
C(N)^T & D(N)
\end{bmatrix} \mid N \in \Mats_r(\K)\biggr\}.$$

Remember that $r \geq 2$.
For $N \in \Mats_r(\K)$, denote by $C_1(N),\dots,C_{n-r}(N)$ the columns of $C(N)$.
With the help of Corollary \ref{generationcor2symmetric}, the same line of reasoning as in Section \ref{section6.1} yields:

\begin{claim}\label{claimRCcar2}
If $r>2$ or $\# \K>2$ then $C_1,\dots,C_{n-r}$ are all range-compatible (and hence, linear). \\
If $r=2$ and $\# \K=2$ then $C_1,\dots,C_{n-r}$ all map every rank $1$ symmetric matrix to a vector of its range.
\end{claim}

From there, we further split the discussion into three subcases.

\vskip 3mm
\noindent \textbf{Case 1.2.1: All the $C_i$ maps are local.} \\
For all $i \in \lcro 1,n-r\rcro$, we find a vector $Y_i \in \K^r$ such that
$C_i : N \mapsto NY_i$. Then, we set
$$Y:=\begin{bmatrix}
Y_1 & \cdots & Y_{n-r}
\end{bmatrix} \in \Mat_{r,n-r}(\K)$$
and
$$Q:=\begin{bmatrix}
I_r & -Y \\
[0]_{(n-r) \times r} & I_{n-r}
\end{bmatrix}.$$
Replacing $\calS$ with $Q^T \calS Q$ leaves all our previous assumptions unchanged,
but in the new situation we have $C=0$.
Finally, for every rank $r$ matrix $N \in \Mats_r(\K)$, we obtain $D(N)=0$.
Since $r \geq 2$, we can use Corollary \ref{generationcor2symmetric} to obtain $D=0$.
We conclude that $\calS=\widetilde{\Mats_r(\K)}^{(n)}$.

\vskip 3mm
\noindent \textbf{Case 1.2.2: All the $C_i$ maps are range-compatible, but one of them is non-local.} \\
Without loss of generality, we can assume that $C_1$ is non-local.
By Theorem \ref{symmetricRC}, we deduce that $\# \K=2$ and that there are vectors
$Y_1,\dots,Y_{n-r}$ of $\K^r$ together with scalars
$a_1,\dots,a_{n-r}$ such that $C_i : N \mapsto NY_i+a_i \Delta(N)$ for all $i \in \lcro 1,n-r\rcro$.
Note that $a_1=1$. With the same line of reasoning as in Case 1.2.1,
we can use a congruence transformation to reduce the situation to the point where $Y_1=\cdots=Y_{n-r}=0$.
Then, setting $L:=\begin{bmatrix}
a_2 & \cdots & a_{n-r}
\end{bmatrix}$,
$R:=\begin{bmatrix}
1 & L \\
[0]_{(n-r-1) \times 1} & I_{n-r-1}
\end{bmatrix}$
and $Q:=I_r \oplus R$, one checks that $Q^T \calS Q$ satisfies all the previous assumptions
with now $C_2=\cdots=C_{n-r}=0$.

It follows that no generality is lost in assuming that $C_1 : N \mapsto \Delta(N)$ and $C_2=\cdots=C_{n-r}=0$.
Now, for any $N \in \Mats_r(\K)$, we can write
$$D(N)=\begin{bmatrix}
a(N) & J(N) \\
J(N)^T & H(N)
\end{bmatrix} \quad \text{with $a(N) \in \K$, $J(N) \in \Mat_{1,n-r-1}(\K)$ and $H(N) \in \Mats_{n-r-1}(\K)$.}$$
If $n>r+1$ then we see that the affine maps $J$ and $H$ vanish at every rank $r$ matrix of $\Mats_r(\K)$, which leads
to $J=0$ and $H=0$ as $r \geq 2$. Hence, in any case we have $J=0$ and $H=0$.
Finally, by extracting the upper-left $(r+1)$ by $(r+1)$ submatrix, we find that,
for every $N \in \Mats_r(\K)$, the matrix
$\begin{bmatrix}
N & \Delta(N) \\
\Delta(N)^T & a(N)
\end{bmatrix}$ is singular.
Computing its determinant leads to
$$\forall N \in \Mats_r(\K), \; \Delta(N)^T N^\ad \Delta(N)=a(N) \det N.$$
Yet, as $\# \K=2$, we remember from Section \ref{section1.2} that
$$\forall N \in \Mats_r(\K), \; \Delta(N)^T N^\ad \Delta(N)=\det(N) \tr(I_{n-1}).$$
Therefore, $a(N)=\tr(I_{n-1})$ for every non-singular matrix $N \in \Mats_r(\K)$.
As $a$ is an affine map and $r \geq 2$, we deduce from Corollary \ref{generationcor2symmetric} that
$a$ is the constant map $N \mapsto \tr(I_{n-1})$.
Hence, $\calS=\widetilde{Z_{r+1}(\K)}^{(n)}$ in that reduced situation.

\vskip 3mm
\noindent \textbf{Case 1.2.3: Some $C_i$ map is not range-compatible.} \\
Then, we know from Claim \ref{claimRCcar2} that $r=2$, $n=3$, that $\# \K=2$ and that $C(0) \neq 0$.
In that case, we shall prove that $\calS$ is congruent to $\calY_1(\K)$ or to $\calY_2(\K)$.

\begin{Rem}\label{subtractlocalremark}
Let $X \in \K^2$, and consider the non-singular matrix $Q:=\begin{bmatrix}
I_2 & X \\
[0]_{2 \times 1} & 1
\end{bmatrix}$. Then, one computes that
$$Q^T \calS Q=\biggl\{\begin{bmatrix}
N & C(N)+NX \\
(C(N)+NX)^T & D(N)+X^TNX
\end{bmatrix} \mid N \in \Mats_2(\K)\biggr\}.$$
Thus, we see that this new space essentially satisfies the same conditions as $\calS$, with $C$
replaced by $N \mapsto C(N)+NX$. In other words, the situation is essentially unchanged
by adding a local map to $C$.
\end{Rem}

We know that $C$ maps the zero matrix to a non-zero vector $X_1$. Without loss of generality, we may assume that $X_1=\begin{bmatrix}
1 \\
0
\end{bmatrix}$ (replacing $\calS$ with $(Q \oplus 1)^T \calS (Q \oplus 1)$ for some well-chosen $Q \in \GL_2(\K)$).
Some local map on $\Mats_2(\K)$ coincides with $C$ on the matrices $\begin{bmatrix}
1 & 1 \\
1 & 1
\end{bmatrix}$ and $\begin{bmatrix}
0 & 0 \\
0 & 1
\end{bmatrix}$, and hence by Remark \ref{subtractlocalremark} we lose no generality in assuming that $C$ vanishes at those two specific matrices.
From there, we discuss whether $C$ maps $E_{1,1}$ to $0$ or to $X_1$.
Note in any case that the four matrices $0, E_{1,1}, E_{2,2}$ and $\begin{bmatrix}
1 & 1 \\
1 & 1
\end{bmatrix}$ generate the affine space $\Mats_2(\K)$, and hence $C$ is uniquely determined by its values on those matrices.

\vskip 3mm
\noindent \textbf{Subcase 1.2.3.1: $C$ maps $E_{1,1}$ to $X_1$.} \\
Then, one checks that
$$C : \begin{bmatrix}
a & b \\
b & c
\end{bmatrix} \longmapsto \begin{bmatrix}
c+1 \\
0
\end{bmatrix}.$$
In that case, writing that the determinant of each matrix of $\calS$ equals zero (with the help of the identity $\forall x \in \K, \; x(x+1)=0$),
we obtain
$$\forall N \in \Mats_2(\K), \; (\det N)\,D(N)=0.$$
Again, the affine map $D$ vanishes at every non-singular matrix of $\Mats_2(\K)$, and hence it equals $0$.
We conclude that $\calS=\calY_1(\K)$.

\vskip 3mm
\noindent \textbf{Subcase 1.2.3.2: $C$ maps $E_{1,1}$ to $0$.} \\
Then, one checks that
$$C : \begin{bmatrix}
a & b \\
b & c
\end{bmatrix} \longmapsto \begin{bmatrix}
a+b+c+1 \\
0
\end{bmatrix}.$$
We find scalars $\alpha,\beta,\gamma,\delta$ such that
$$D : \begin{bmatrix}
a & b \\
b & c
\end{bmatrix} \longmapsto \alpha a+\beta b+\gamma c+\delta.$$
Writing that every matrix of $\calS$ has determinant $0$ and using the identity $\forall x \in \K, \; x^2=x$, we obtain
$$\forall (a,b,c)\in \K^3, \; \beta\, abc+(\alpha+\gamma+\delta+1)\,ac+(1+\gamma)\,bc+\alpha\, ab+(\beta+\delta)\,b=0.$$
On the left hand-side of this equality we have a polynomial of degree at most 1 in each variable, and hence
its coefficients equal zero. This leads to $\beta=\alpha=0$, $\gamma=1$ and $\delta=\beta$.
We conclude that $\calS=\calY_2(\K)$.

\vskip 3mm
This completes our investigation of the case when $m=0$.

\subsection{Case 2: $0<m \leq s$ and $S_H$ contains a non-zero alternating matrix.}

In that case, no generality is lost in assuming that $S_H$ contains $E_{1,n}+E_{n,1}$.
Then, we split every matrix $M \in \calS$ up as
$$M=\begin{bmatrix}
P(M) & [?]_{(n-1) \times 1} \\
[?]_{1 \times (n-1)} & ?
\end{bmatrix} \quad \text{with $P(M) \in \Mats_{n-1}(\K)$}$$
and one splits $P(M)$ further up as
$$P(M)=\begin{bmatrix}
? & [?]_{1 \times (n-2)} \\
[?]_{(n-2) \times 1} & K(M)
\end{bmatrix} \quad \text{with $K(M) \in \Mats_{n-2}(\K)$.}$$
Then, with the line of reasoning from Section \ref{section7.2}, we combine our assumptions on $\calS$ with the inequality statement in Theorem
\ref{maintheosymmetric} and Lemma \ref{symmetriccornerlemma1} to obtain the following facts:
\begin{itemize}
\item The set $K(\calS)$ is an affine subspace of $\Mats_{n-2}(\K)$ with upper-rank at most $r-2$
and dimension $s^{(2)}_{n-2,r-2}$;
\item The space $\calS$ has dimension $s^{(2)}_{n,r}$;
\item Either $r \in \{2,3\}$ or $s^{(2)}_{n-2,r-2}>s^{(1)}_{n-2,r-2}$;
\item $m=s$.
\end{itemize}
In any case, by induction we recover that:
\begin{itemize}
\item Either $K(\calS)$ is congruent to $\WS_{n-2,r-2}(\K)$;
\item Or $r-2$ is odd, $\# \K=2$ and $K(\calS)$ is congruent to $Z'_{n-2,r-2}(\K)$.
\end{itemize}
Performing an additional congruence on $\calS$, we lose no generality in assuming that:
\begin{itemize}
\item Either $K(\calS)=\WS_{n-2,r-2}(\K)$;
\item Or $r$ is odd, $\# \K=2$ and $K(\calS)=Z'_{n-2,r-2}(\K)$.
\end{itemize}
Then, with the line of reasoning from Section \ref{section7.2}, it follows that:
\begin{itemize}
\item Either $P(\calS)=\WS_{n-1,r}(\K)$;
\item Or $r$ is odd, $\# \K=2$ and $P(\calS)=Z'_{n-1,r}(\K)$.
\end{itemize}
From there, we split the discussion into three subcases:
\begin{itemize}
\item Subcase 2.1: $r$ is even;
\item Subcase 2.2: $r$ is odd and $P(\calS)=\WS_{n-1,r}(\K)$;
\item Subcase 2.3: $r$ is odd, $\# \K=2$ and $P(\calS)=Z'_{n-1,r}(\K)$.
\end{itemize}
Note in any case that by the inequality statement in Theorem \ref{maintheosymmetric}, we have
$$s_{n,r}^{(2)} \geq s_{n,r.}^{(1)}$$

\subsubsection{Subcase 2.1: $r$ is even.}

We wish to prove that $\calS$ is congruent to $\WS_{n,r}(\K)$ or that
$n=3$, $r=2$, $\# \K=2$ and $\calS$ is congruent to $\calY_3(\K)$.

Unless $n \in \{3,4\}$, the line of reasoning featured in Section \ref{section7.2.1} can be transposed effortlessly so as to yield
that $\calS$ is congruent to $\WS_{n,r}(\K)$ (indeed, in that case we have $n-s-1 \geq s+1$, and either $n-s-1>s+1$ or $s+1>2$,
which helps validate the assumptions of Corollary \ref{Flanderscor} for the $\calT$ space).

Next, assume that $n=4$, so that $r=2$. For all $M \in \calS$, let us write
$$M=\begin{bmatrix}
A(M) & B(M)^T & ? \\
B(M) & [0]_{2 \times 2} & C(M) \\
? & C(M)^T & a(M)
\end{bmatrix}$$
where $A(M)$ and $a(M)$ are scalars, $B(M) \in \K^2$ and $C(M) \in \K^2$.
Again, we consider the $\calT$ space defined as
$$\calT:=\Bigl\{\begin{bmatrix}
B(M) & C(M)
\end{bmatrix} \mid M \in \calS\Bigr\},$$
to which we want to apply Corollary \ref{Flanderscor}.
Since $\urk \calT \leq 1$ and $B(\calS)=\K^2$, we know from Flanders's theorem that $\dim \calT=2$,
and hence $C(M)$ is an affine function of $B(M)$ only.
If there existed $M \in \calS$ such that $B(M)=0$ and $C(M)\neq 0$, then
as $P(\calS)=\WS_{3,2}(\K)$ we would even find such a matrix with $A(M)=1$, and it is obvious that we would have
$\rk M=3$, contradicting the fact that $r=2$. Hence, $\calT$ contains the zero matrix, and
Corollary \ref{Flanderscor} applies to it. From there, the line of reasoning from Section \ref{section7.2.1} can be applied
effortlessly and it yields that $\calS$ is congruent to $\WS_{4,2}(\K)$.

\vskip 3mm
It remains to tackle the case when $n=3$. Remembering that $S$ contains $E_{1,3}+E_{3,1}$ and that $m=s=1$, we can find
scalars $\alpha,\beta,\gamma,\lambda,\mu,\nu$ such that
$$\calS=\Biggl\{\begin{bmatrix}
a & b & c \\
b & 0 & \alpha a+\beta b+\gamma \\
c & \alpha a+\beta b+\gamma & \lambda a+\mu b+\nu
\end{bmatrix} \mid (a,b,c) \in \K^3\Biggr\}.$$
Setting $Q:=\begin{bmatrix}
1 & 0 & \beta \\
0 & 1 & 0 \\
0 & 0 & 1
\end{bmatrix}$ and replacing $\calS$ with $Q^T \calS Q$, we see that no generality is lost in assuming that $\beta=0$.
If $\mu \neq 0$, then the linear hyperplane $H_2:=\K \times \{0\} \times \K$ would satisfy $S_{H_2}=\{0\}$,
contradicting $m=1$. Thus, $\mu=0$.
If $\alpha=\gamma=\nu=\lambda=0$, then $\calS=\WS_{3,2}(\K)$.

Assume now that $(\alpha,\gamma,\nu,\lambda) \neq (0,0,0,0)$.
Computing the determinant of the matrices of $\calS$ yields
$$\forall (a,b,c)\in \K^3, \; \alpha^2 a^3+\lambda ab^2+\nu b^2+\gamma^2 a=0.$$
If $\# \K>2$, then $\# \K>3$ and hence the polynomial on the left hand-side, whose total degree is less than $4$,
is formally zero, which would contradict our assumptions. Therefore, $\# \K=2$.
From there, we obtain
$$\forall (a,b)\in \K^2, \; \lambda ab+(\alpha+\gamma)a+\nu b=0.$$
This yields $\lambda=\nu=0$ and $\alpha=\gamma$. Thus, $\alpha=\gamma=1$, and we conclude that $\calS=\calY_3(\K)$.

This finishes the proof in the subcase when $r$ is even.

\subsubsection{Subcase 2.2: $r$ is odd and $P(\calS)=\WS_{n-1,r}(\K)$.}

For any $M \in \calS$, let us write
$$M=\begin{bmatrix}
A(M) & B_1(M)^T & B(M)^T & [?]_{s \times 1} \\
B_1(M) & a(M) & [0]_{1 \times (n-s-2)} & b(M) \\
B(M) & [0]_{(n-s-2) \times 1} & [0]_{(n-s-2) \times (n-s-2)} & C(M) \\
[?]_{1 \times s} & b(M) & C(M)^T & c(M)
\end{bmatrix}$$
where $a(M)$, $b(M)$ and $c(M)$ are scalars, $A(M) \in \Mats_s(\K)$, $B_1(M) \in \Mat_{1,s}(\K)$, $B(M) \in \Mat_{n-s-2,s}(\K)$
and $C(M) \in \K^s$.

Then, as in Section \ref{section7.2.2}, we obtain that the affine space
$$\calT:=\Bigl\{\begin{bmatrix}
B(M) & C(M)
\end{bmatrix} \mid M \in \calS\Bigr\}$$
satisfies $\urk \calT \leq s$.
We want to prove that there exists a vector $Y \in \K^s$ such that $C(M)=B(M)Y$ for all $M \in \calS$.
With the same line of reasoning as in Section \ref{section7.2.2}, this would follow from Corollary \ref{Flanderscor}
unless $\# \K=2$, $n-s-2=s+1=2$ and $\calT$ does not contain the zero matrix.
Assume that this exceptional case holds. Then, $s=1$ and $n=5$. By Flanders's theorem, we have $\dim \calT \leq 2$, and
as $B(\calS)=\K^2$ we deduce that $C(M)$ is an affine function of $B(M)$ only.
Using $P(\calS)=\WS_{n-1,r}(\K)$, it follows that we can choose $M_0 \in \calS$ such that
$$\begin{bmatrix}
A(M_0) & B_1(M_0)^T \\
B_1(M_0) & a(M_0)
\end{bmatrix}=I_2, \quad B(M_0)=0 \quad \text{and} \quad C(M_0) \neq 0,$$
and it would follow that $\rk M_0 \geq 4$, contradicting the fact that $r=3$.

Thus, Corollary \ref{Flanderscor} applies to $\calT$.
Then, by following the chain of arguments from Section \ref{section7.2.2}, we arrive, after a harmless congruence transformation,
to the point where $C=0$ and where, by setting
$$J(M):=\begin{bmatrix}
a(M) & b(M) \\
b(M) & c(M)
\end{bmatrix}$$
for all $M \in \calS$, the space $J(\calS)$ is a $1$-dimensional affine subspace of $\Mats_2(\K)$ with upper-rank at most $1$.
In order to conclude, we can use Proposition \ref{rankatmost1lemma} to recover the possible structures of $J(\calS)$:
\begin{itemize}
\item Either $J(\calS)$ is congruent to $\Vect(E_{1,1})$, and then $\calS$ is congruent to a subspace
of $\WS_{n,r}(\K)$; as $\dim \calS=s_{n,r}^{(2)}=\dim \WS_{n,r}(\K)$, it would follow that $\calS$ is congruent to
$\WS_{n,r}(\K)$.
\item Or $\# \K=2$ and $J(\calS)$ is congruent to $Z_2(\K)$; in that case $\calS$ is congruent to a subspace of $Z'_{n,r}(\K)$, and as
those two spaces share the same dimension we conclude that $\calS$ is congruent to $Z'_{n,r}(\K)$.
\end{itemize}
In any case, we have obtained one of the desired outcomes.

\subsubsection{Subcase 2.3: $r$ is odd, $\# \K=2$ and $P(\calS)=Z'_{n-1,r}(\K)$.}

Now, we split every matrix of $\calS$ up as
$$M=\begin{bmatrix}
A(M) & B(M)^T & [?]_{s \times 1} \\
B(M) & R(M) & C(M) \\
[?]_{1 \times s} & C(M)^T & b(M)
\end{bmatrix}$$
with $A(M) \in \Mats_s(\K)$, $B(M) \in \Mat_{n-1-s,s}(\K)$,
$C(M) \in \K^s$, $b(M) \in \K$ and
$$R(M)=\begin{bmatrix}
K(M) & [0]_{2 \times (n-3-s)} \\
[0]_{(n-3-s) \times 2} & [0]_{(n-3-s) \times (n-3-s)}
\end{bmatrix} \quad \text{with $K(M) \in Z_2(\K)$.}$$
Then, we split
$$C(M)=\begin{bmatrix}
C_1(M) \\
C_2(M)
\end{bmatrix}
\quad \text{with $C_1(M) \in \K^2$ and $C_2(M) \in \K^{n-3-s}$}$$
and
$$B(M)=\begin{bmatrix}
B_1(M) \\
B_2(M)
\end{bmatrix} \quad \text{with $B_1(M) \in \Mat_{2,s}(\K)$ and $B_2(M) \in \Mat_{n-3-s,s}(\K)$.}$$
Set
$$\calT:=\Bigl\{\begin{bmatrix}
B_2(M) & C_2(M)
\end{bmatrix}\mid M \in \calS\Bigr\}\subset \Mat_{n-3-s,s+1}(\K).$$
As $K(M)$ has rank $1$ for all $M \in \calS$, we find that $\urk \calT \leq s$.

As $P(\calS)=Z'_{n-1,r}(\K)$, we find that $B_2(\calS)=\Mat_{n-3-s,s}(\K)$.
Then, Corollary \ref{Flanderscor} applies to $\calT$ except in two exceptional situations:
\begin{itemize}
\item If $(n-3-s,s+1,\# \K)=(2,2,2)$ and $\calT$ does not contain the zero matrix.
\item If $n-3-s<s+1$.
\end{itemize}
Assume that the first exceptional case holds, so that $s=1$ and $n=6$.
As $B_2(\calS)=\K^2$, we would deduce from Flanders's theorem that $C_2(M)$ is an affine function of $B_2(M)$ only.
As $P(\calS)=Z'_{n-1,r}(\K)$, it would follow that there is a matrix $M \in \calS$ such that $B_2(M)=0$, $C_2(M) \neq 0$,
and whose upper-left $3$ by $3$ block equals $\begin{bmatrix}
1 & 0 & 0 \\
0 & 0 & 0 \\
0 & 0 & 1
\end{bmatrix}$. Yet, this would yield $\rk M \geq 4$, contradicting $\urk \calS \leq 3$.

Assume now that the second exceptional case holds, so that $n \leq 2s+3$. As $s^{(2)}_{n,r} \geq s^{(1)}_{n,r}$,
we have $5s \leq 2n-5$, and we easily obtain that $(n,s)=(5,1)$.
Assume that $C_2(M)$ is not a linear function of $B_2(M)$. Then,
$$\calS':=\bigl\{M \in \calS : \; B_2(M)=0 \quad \text{and} \quad C_2(M)=1\bigr\}$$
is an affine subspace of $\calS$ with codimension at most $2$.
For $M \in \calS$, denote its upper-left $3$ by $3$ block by $H(M) \in \Mats_3(\K)$.
With a standard rank computation, we see that $\rk H(M) \leq 1$ for all $M \in \calS'$.
Yet, as $P(\calS)=Z'_{n-1,r}(\K)$ we see that $\dim H(\calS)=4$, whence $\dim H(\calS')\geq 2$.
By the first statement in Theorem \ref{maintheosymmetric}, we obtain a contradiction.

Thus, in any case, either we can apply Corollary \ref{Flanderscor} to $\calT$ or we directly have
that $C_2(M)$ is the product of $B_2(M)$ with a fixed scalar. Thus, with an additional congruence, we
reduce the situation to the one where $C_2=0$.
Finally, for $M \in \calS$, we can set
$$J(M)=\begin{bmatrix}
K(M) & C_1(M) \\
C_1(M)^T & b(M)
\end{bmatrix}\in \Mats_3(\K).$$
Using the minimality of $m$, we obtain, just like in Section \ref{section6.2},
that $S$ contains $E_{i,j}+E_{j,i}$ for all $i \in \lcro 1,s\rcro$ and all $j \in \lcro s+3,n-1\rcro$.
Moreover, combining inequality $5s \leq 2n-5$ with $s \geq 1$, we see that $n-3-s \geq s$.
As in the previous cases, we can then use Lemma \ref{symmetriccornerlemma1} inductively
to obtain that $\rk J(M) \leq 1$ for all $M \in \calS$.
Thus, $J(\calS)$, which is not a linear subspace of $\Mats_3(\K)$, has upper-rank at most $1$.
On the other hand $\dim J(\calS) \geq 1$ since $P(\calS)=Z'_{n-1,r}(\K)$.
Thus, by Proposition \ref{rankatmost1lemma}, the space $J(\calS)$ is congruent to $\widetilde{Z_2(\K)}^{(3)}$.
From there, we conclude that $\calS$ itself is congruent to a subspace of $Z'_{n,r}(\K)$. Since both spaces have dimension $s_{n,r}^{(2)}$
 we conclude that $\calS$ is congruent to $Z'_{n,r}(\K)$.

\vskip 3mm
This completes the study of Case 2.

\subsection{Case 3: $m=1$ and $S_H$ contains a non-alternating matrix.}

With the line of reasoning from Section \ref{section5.1.3}, we obtain an affine subspace
$\calT$ of $\Mats_{n-1}(\K)$ such that $\dim \calT=\dim \calS-1$ and $\urk \calT \leq r-1$.
From there, the chain of arguments of Section \ref{section7.3} can be followed effortlessly
so as to obtain that $\dim \calS <s_{n,r}^{(1)}$ or $\dim \calS <s_{n,r}^{(2)}$, thereby contradicting our assumptions.

\subsection{Case 4: $s < m$ and all the matrices in $S_H$ are alternating.}

We shall prove that this case leads to a contradiction.

First of all, we know from the start of the proof that since $s<m$ we must have
$r=n-1$ and $m \leq n-2$.
Then, $\left \lfloor \frac{n-1}{2}\right \rfloor <n-2$, leading to $n \geq 4$.

Moreover, since $r=n-1$ and $n \geq 4$, we see from Remark \ref{dimensionremarksymmetric} that
$$\dim \calS=\dbinom{n}{2}.$$
With a harmless congruence transformation, we see that no generality is lost in assuming that
$$S_H=\biggl\{\begin{bmatrix}
[0]_{(n-1)\times (n-1)} & X \\
X^T & 0
\end{bmatrix} \mid X \in \K^m \times \{0\}\biggr\}.$$
Then, we split every matrix $M$ of $\calS$ up as
$$M=\begin{bmatrix}
a(M) & L(M) & ? \\
L(M)^T & K(M) & [?]_{(n-2) \times 1} \\
? & [?]_{1 \times (n-2)} & ?
\end{bmatrix},$$
with $a(M) \in \K$, $L(M) \in \Mat_{1,n-2}(\K)$ and $K(M)\in \Mats_{n-2}(\K)$.

Since $S_H$ contains $E_{1,n}+E_{n,1}$, Lemma \ref{symmetriccornerlemma1} yields
$\urk K(\calS) \leq n-3$. Hence,
$$\dim \calS \leq \dim K(\calS)+(n-1)+\dim S_H \leq \dbinom{n-2}{2}+(n-1)+(n-2)=\dbinom{n}{2}.$$
As $\dim \calS=\dbinom{n}{2}$, all the intermediate inequalities turn out to be equalities, which yields:
\begin{enumerate}[(i)]
\item $m=n-2$;
\item $\dim K(\calS)=\dbinom{n-2}{2}$;
\item For all $a_1 \in \K$ and $L_1 \in \Mat_{1,n-2}(\K)$, the space $\calS$ contains a matrix of the form
$$\begin{bmatrix}
a_1 & L_1 & ? \\
L_1^T & [0]_{(n-2) \times (n-2)} & [?]_{(n-2) \times 1} \\
? & [?]_{1 \times (n-2)} & ?
\end{bmatrix}.$$
\end{enumerate}
Next, applying the same extraction technique but starting now from $E_{k,n}+E_{n,k}$
for an arbitrary $k \in \lcro 2,n-2\rcro$, we recover that
the translation vector space of $K(\calS)$ contains every symmetric matrix of the form
$$\begin{bmatrix}
[?]_{(n-3) \times (n-3)} & [?]_{(n-3) \times 1} \\
[?]_{1 \times (n-3)} & 0
\end{bmatrix}.$$
Then, $\dim K(\calS) \geq \dbinom{n-1}{2}-1$. Using statement (ii) above, we deduce that
$n-2 \leq 1$. This contradicts the fact that $n \geq 4$.

\subsection{Case 5: $s < m$ and $S_H$ contains a non-alternating matrix.}

We shall prove that this case leads to a final contradiction.

As in the previous case, the fact that $s<m$ leads to $r=n-1$, $m \leq n-1$ and $\dim \calS=\dbinom{n}{2}$.
In particular, $n \geq 3$ and $m \geq 2$.

Let us split any matrix $M \in \calS$ up as
$$M=\begin{bmatrix}
P(M) & C(M) \\
C(M)^T & a(M)
\end{bmatrix}$$
with $P(M) \in \Mats_{n-1}(\K)$, $C(M) \in \K^{n-1}$ and $a(M) \in \K$.

In $S_H$, we can find a matrix of the form
$$N=\begin{bmatrix}
[0]_{(n-1) \times (n-1)} & C_0 \\
C_0^T & 1
\end{bmatrix} \quad \text{with $C_0 \in \K^{n-1}$.}$$
Set $\calT:=C_0C_0^T+P(\calS)$.
By Lemma \ref{symmetriccornerlemma3}, we find that
$\urk(\calT) \leq n-2$, and hence by the inequality statement from Theorem \ref{maintheosymmetric} (together with the remark that
$s_{n-1,n-2}^{(2)} \leq s_{n-1,n-2}^{(1)}$),
$$\dim \calT \leq \dbinom{n-1}{2}.$$
Then, by the rank theorem,
$$\dim \calS =\dim P(\calS)+\dim S_H= \dim \calT+\dim S_H \leq \dbinom{n-1}{2}+(n-1)=\dbinom{n}{2}.$$
Thus, all the previous inequalities turn out to be equalities, which leads to:
$$\dim \calT=\dbinom{n-1}{2} \quad \text{and} \quad m=n-1.$$
Then, $\calT$ is an affine subspace of singular matrices of $\Mats_{n-1}(\K)$ with the maximal dimension, and hence we can
apply the induction hypothesis to it.
Noting that $s_{n-1,n-2}^{(1)}>s_{n-1,n-2}^{(2)}$ unless $n-1=3$, we lose no generality in assuming that one of the following six
 cases holds:
\begin{enumerate}[(a)]
\item $\calT=\widetilde{\Mats_{n-2}(\K)}^{(n-1)}$;
\item $\calT=\Mata_{n-1}(\K)$ and $n$ is even;
\item $n=4$ and $\calT=\WS_{3,2}(\K)$;
\item $\K$ has cardinality $2$ and $\calT=Z_{n-1}(\K)$;
\item $\K$ has cardinality $2$ and either $\calT=\calY_1(\K)$ or $\calT=\calY_2(\K)$;
\item $\K$ has cardinality $2$ and $\calT=\calY_3(\K)$.
\end{enumerate}

Now, let us consider the linear hyperplane $H'$ of $\K^n$ defined by the equation $x_{n-1}=0$ in the standard basis.
Noting that $\calT$ has the same translation vector space as $P(\calS)$,
we deduce that $\dim S_{H'} \leq 1$ in any one of cases (a), (d) and (e), whence $H'$ is $\calS$-adapted,
and we contradict the fact that $m \geq 2$.
In cases (c) and (f), we obtain that $\dim S_{H'} \leq 2$ and that all the matrices in $S_{H'}$ are alternating;
since $n \geq 4$, this shows that $H'$ is $\calS$-adapted; yet, $m\geq 2$, whence $\dim S_{H'}=2$; some congruence transformation then
turns $\calS$ into a space that satisfies Case 4, a situation which has been shown to yield a contradiction.

\vskip 3mm
Hence, case (b) holds whatever the choice of the matrix $N$ we have started from.
In particular, the translation vector space of $P(\calS)$ equals $\Mata_{n-1}(\K)$.
Choose again $N \in S_H$ such that $a(N)=1$. Let $N' \in S_H$ be such that $a(N')=0$.
As $a(N+N')=1$ we deduce that the matrix $D:=C(N+N')^T C(N+N')-C(N)C(N)^T$
is alternating. Yet, $C(N)C(N')^T+C(N')C(N)^T$ is obviously symmetric with diagonal zero.
Thus, expanding the expression of $D$ shows that $C(N')C(N')^T$ is alternating.
Considering the diagonal entries yields $C(N')=0$ and hence $N'=0$.
Thus, $S_H$ contains no non-zero alternating matrix, thereby contradicting the assumption that $m \geq 2$.

\vskip 3mm
The last remaining case has been shown to yield a contradiction, and our proof of Theorem \ref{maintheosymmetric}
is finally complete!

\end{document}